\documentclass{article}
\pdfoutput=1

\PassOptionsToPackage{numbers}{natbib}

\usepackage[preprint]{nips_2018}

\usepackage[utf8]{inputenc} 
\usepackage[T1]{fontenc}    
\usepackage{hyperref}       
\usepackage{url}            
\usepackage{booktabs}       
\usepackage{amsfonts}       
\usepackage{nicefrac}       
\usepackage{microtype}      
\usepackage{dblfloatfix}
\usepackage[mathscr]{euscript}
\RequirePackage{amsthm,amsmath}
\usepackage{color}
\usepackage{stmaryrd}
\usepackage{multirow}
\usepackage[titletoc,title]{appendix}
\usepackage{bbm}
\usepackage{graphicx}
\usepackage{enumitem}
\usepackage{algorithm,algpseudocode}
\usepackage{subcaption}
\usepackage{mathabx}
\usepackage{rotating}
\usepackage{amsmath}
\usepackage[capitalise]{cleveref}
\usepackage[usenames,dvipsnames]{xcolor}
\definecolor{purple}{rgb}{0.3,0.0,.4}

\DeclareMathOperator{\var}{Var}

\numberwithin{equation}{section}
\theoremstyle{plain}
\newtheorem{thm}{Theorem}[section]

\newtheorem{lem}[thm]{Lemma}

\newtheorem*{remark}{Remark}
\newtheorem{definition}{Definition}[section]

\title{Tensor Random Projection for Low Memory Dimension Reduction}

\author{
Yiming Sun \thanks{Both authors contributed equally.} \\
Cornell \\
\And Yang Guo \footnotemark[1]\\
UW-Madison\\
\And Joel A.~Tropp \\
Caltech\\
\And Madeleine Udell\\
Cornell \\
}

\begin{document}

\maketitle

\begin{abstract}
Random projections reduce the dimension of a set of vectors
while preserving structural information,
such as distances between vectors in the set.
This paper proposes a novel use of row-product random matrices \cite{rudelson2012row} in random projection, where we call it Tensor Random Projection (TRP). It requires substantially less memory than existing dimension reduction maps.
The TRP map is formed as the Khatri-Rao product of several smaller
random projections, and is compatible with any base random projection
including sparse maps, which enable dimension reduction with
very low query cost and no floating point operations.
We also develop a reduced variance extension.
We provide a theoretical analysis of the bias and variance of the TRP,
and a non-asymptotic error analysis for a TRP composed of two smaller maps.
Experiments on both synthetic and MNIST data show that
our method performs as well as conventional methods with substantially less storage.
\end{abstract}

\section{Introduction}
Random projections (RP) are commonly used to reduce the dimension of collections of
high dimensional vectors, enabling a broad range of modern applications
\cite{wright2009robust,buhler2002finding,allen2014sparse,bingham2001random,fradkin2003experiments, halko2011finding}.
In the context of large-scale relational databases, these maps enable
applications like information retrieval \cite{papadimitriou2000latent},
similarity search \cite{sahin2005prism,kaski1998dimensionality},
and privacy preserving distributed data mining \cite{liu2006random}. Consider the problem of detecting plagiarism.
We might attempt to solve this problem by comparing the similarity of
word-level n-gram profiles for different pairs of documents \cite{barron2009automatic}.
To avoid tremendous query cost of this procedure,
which scales quadratically with the number of documents,
we may instead reduce the dimension of the data vector with a random projection,
and cluster the resulting low-dimensional vectors.
However, if the dimension of the vectors before reduction
(here, the size of the lexicon) is too big,
the storage cost of the random map is not negligible.
Furthermore, even generating the pseudo-random numbers used to produce the random projection
is expensive \cite{matsumoto1998mersenne}.

To reduce the storage burden, we propose a novel use of the row-product random matrices in random projection, and call it the \textit{Tensor Random Projection} (TRP),
formed as the Khatri-Rao product of a list of smaller dimension reduction maps.
We show this map is an approximate isometry, with tunable accuracy,
and hence can serve as a useful dimension reduction primitive.
Furthermore, the storage required to compress $d$ dimension vectors scales as $\sqrt[N]d$
where $N$ is the number of smaller maps used to form the TRP.
We also develop a reduced variance version of the TRP that allows separate control
of the dimension of the range and the quality of the isometry.

\paragraph{Dimension Reduction Map}
A function $f$ from $\mathbb{R}^d \rightarrow \mathbb{R}^k (k\ll d)$ is
called a dimension reduction map (DRM) if it approximately preserves the pair-wise distances.
More precisely,
we call $f$ a \emph{$\epsilon$-Johnson-Lindenstrauss (JL) transform} if
for any $\epsilon>0$ and for any two points $\mathbf{u},\mathbf{v}$ in a discrete set
$\mathcal{X} \subseteq \mathbb{R}^n$, we have
\[
(1-\epsilon)\|\mathbf{u} -\mathbf{v}\|^2  \leq \|f(\mathbf{u}) - f(\mathbf{v}) \|^2 \leq  (1+\epsilon) \|\mathbf{u}- \mathbf{v}\|^2
\]
The well-known JL Lemma \cite{johnson1984extensions} claims for $k = O(\log(|\mathcal X|)/\epsilon^2)$,
an $\epsilon$-JL transform exists.
In fact, the proof shows that a suitable random linear map is an $\epsilon$-JL transform with high probability.
\par

The simplicity of linear maps makes them a favorite choice for dimension reduction.
A linear map
$f(\mathbf{x}) = \mathbf{Ax}$ for $\mathbf{A}\subseteq{\mathbb{R}^{d \times k}}$
is a good DRM if has the following properties:
\begin{enumerate}
\item \label{expected-metry} \emph{Expected Isometry.}
In expectation, the map $A$ is an isometry: $\mathbb{E}\|\mathbf{Ax}\|^2 = \|\mathbf{x}\|^2$.
\item \label{vanishing-variance} \emph{Vanishing Variance.}
$\textrm{Var}(\|\mathbf{Ax}\|^2)$ decays to zero as $k$ increases.
The variance measures the deviations from isometry, and serves
as a quality metric for the DRM.
\item \emph{Database-Friendly\label{db-friendly}.}
A map is database-friendly if it uses not-too-much storage (and so fits in memory),
can be applied to a vector with relatively few queries to vector elements
(and so uses few database lookups),
and is computationally cheap to construct and apply.

\end{enumerate}
Lemma \ref{lemma:inner-product} in Appendix \ref{sec:appendix_proof} shows
any linear map that is an expected isometry with vanishing variance
is a $\epsilon-$JL transform with high probability, for sufficiently large $k$.

Sparse random maps for low memory dimension reduction
were first proposed by \cite{achlioptas2003database},
and further work has improved the memory requirements and guarantees of these methods
\cite{li2006very,bourgain2015toward}.
Most closely related to our work is Rudelson's foundational study \cite{rudelson2012row},
which considers how the spectral and geometric properties of
the random maps we use in this paper resemble a random map with iid entries,
and shows that their largest and smallest singular values are of the same order.
These results have been widely used to obtain guarantees for algorithmic privacy,
but not for random projection.
Battaglino et al. \cite{battaglino2018practical} use random projections
of Khatri-Rao products to develop a randomized least squares algorithm
for tensor factorization;
in contrast, our method uses the (full) Khatri-Rao product to enable random projection.
Sparse random projections to solve least squares problems were
also explored in \cite{wang2015fast} and \cite{woodruff2014sketching}.
To our knowledge, this paper is the first to consider using the Khatri-Rao product
for low memory random projection.









\subsection{Notation}
We denote \textit{scalar}, \textit{vector}, and \textit{matrix} variables, respectively,
by lowercase letters ($x$), boldface lowercase letters ($\mathbf{x}$),
and boldface capital letters  ($\mathbf{X}$).
Let $[N] = \{1, \dots, N\}$.
For a vector $\mathbf{x}$ of size n, we let $\|\mathbf{x}\|_q = (\sum_{i=1}^n x_i^q)^{1/q}$
be its $q$ norm for $q\ge 1$.
For a matrix $\mathbf{X}$, we denote its $i^{th}$ row, $j^{th}$ column,
and the $(i,j)^{th}$ element as $\mathbf{X}(i,.)$, $\mathbf{X}(.,j)$, and $\mathbf{X}(i,j)$.
We let $\mathbf{A} \odot \mathbf{B}$ denote the \textit{Khatri-Rao product},
$\mathbf{A} \in \mathbb{R}^{I \times K}, \mathbf{B} \in \mathbb{R}^{J \times K}$,
i.e. the ``matching column-wise'' Kronecker product.
The resulting matrix of size $(IJ) \times K$ is given by:
\begin{equation}\label{khatri-rao}
\mathbf{A} \odot \mathbf{B} = \left[
\begin{array}{ccc}
\mathbf{A}(1,1)\mathbf{B}(\cdot,1)   & \cdots & \mathbf{A}(1,K)\mathbf{B}(\cdot,K) \\
\vdots & \ddots & \vdots \\
\mathbf{A}(I,1)\mathbf{B}(\cdot, 1) & \cdots &   \mathbf{A}(I,K)\mathbf{B}(\cdot,K)
\end{array}
\right].
\end{equation}

\section{Tensor Random Projection}
We seek a random projection map to embed a collection of vectors
$\mathcal X \subseteq \mathbb{R}^{d}$
into $\mathbb{R}^k$ with $k \ll d$.
Let us take $d = \prod_{n=1}^N d_n$, motivated by the problem of compressing
(the vectorization of) an order $N$ tensor with dimensions $d_1,\ldots,d_N$. Conventional random projections use $O(kd)$ random variables.
Generating so many random numbers is costly; and storing them can be costly when $d$ is large.
Is so much randomness truly necessary for a random projection map?

To reduce randomness and storage requirements, we propose
the \emph{tensor random projection} (TRP):
\begin{equation}
\label{eq:TRP}
f_{\text{TRP}}(\mathbf{x}):= (\mathbf{A}_1 \odot \cdots \odot \mathbf{A}_N)^\top
\mathbf{x},
\end{equation}
where each $\mathbf{A}_i \in \mathbb{R}^{d_i \times k}$, for $i \in [N]$,
can be an arbitrary RP map and
$\mathbf{A} := (\mathbf{A}_1 \odot \cdots \odot \mathbf{A}_N)^\top$.
We call $N$ the \emph{order} of the TRP.
We show in this paper that the TRP is an expected isometry,
has vanishing variance,
and supports database-friendly operations.

The TRP requires only $k\sum_{i = 1}^N d_i$ random variables
(or $k\sqrt[N]{d}$ by choosing each $d_i$ to be equal),
rather than the $kd$ random variables needed by conventional methods.
Hence the TRP is database friendly:
it significantly reduces storage costs and randomness requirements compared to its
constituent DRMs.

In large scale database settings,
where computational efficiency is critical and queries of vector elements are costly,
practitioners often use sparse RPs.
Let $\delta$ be the proportion of non-zero elements in the RP map.
To achieve a $\delta$-sparse RP, a common construction is the scaled sign random map:
each element is distributed as $(-1/\sqrt{\delta}, 0, 1/\sqrt{\delta})$ with probability
$(\delta/2, 1-\delta, \delta/2)$.
\cite{achlioptas2003database} proposed $\delta=1/3$,
while \cite{li2006very} further suggests a sparser scheme with
$\delta=1/\sqrt{d}$ that he calls the \textit{Very Sparse} RP.

To further reduce memory requirements of random projection,
we can form a TRP whose constituent submatrices
are generated each with sparsity factor $\delta$,
which leads to a $\delta^N$-sparse TRP. Under sparse setting, it is a $(1/3)^N$ sparse TRP while under very sparse setting, it is a $1/\sqrt{d}$ sparse TRP. 
Both TRPs can be applied to a vector using very few queries to vector elements
and no multiplications.
Below, we show both sparse and very sparse TRP are low-variance approximate isometry empirically. 

\paragraph{Variance Reduction}
One quirk of many DRMs is that the variance of the map is controlled by the range $k$ of the map.
However, with the TRP we can reduce the variance without increasing $k$.
We propose the \rm{TRP(T)},
a scaled average of $T$ independent TRPs we call \emph{replicates},
defined as
\begin{equation}
f_{\text{TRP(T)}}(\mathbf{x}):= \frac{1}{\sqrt{T}}\sum_{t = 1}^T f^{(t)}_{\text{TRP} }(\mathbf{x}).
\end{equation}
(Note that the average of $T$ TRPs is not itself a TRP.)
We discuss theoretical properties of this map in the main theory section below.

\section{Main Theory}
In this section, we will show the TRP and TRP(T) are expected isometries with vanishing variance.
We provide a rate for the decrease in variance with $k$.
We also prove a non-asymptotic concentration bound on the quality of the isometry when $N=2$.
Without loss of generality, we state our results only for the TRP(T),
since the TRP follows as a special case with $T=1$. We begin by showing the TRP(T) is an approximate isometry.
\begin{lem}
\label{lemma: norm-preserve}
Fix $\mathbf{x} \in \mathbb{R}^{\prod_{n=1}^N d_n}$.
Form a \textup{TRP}(T) of order $N$
composed of $k$ independent matrices
whose columns are independent random vectors of mean zero in isotropic positions, 
i.e. with identity covariance matrix. Then,
\begin{equation}
\label{eq:lemma-invariant-length-statement}
\mathbb{E} \|f_{\textup{TRP(T)}}(\mathbf{x})\|^2 = \|\mathbf{x}\|^2. \nonumber
\end{equation}
\end{lem}
Interestingly, \cref{lemma: norm-preserve} does not require elements of $\mathbf{A}_n$ to be i.i.d..
Now we present an explicit form for the variance of the isometry.
\begin{lem}
\label{lemma:variance}
Fix $\mathbf{x} \in \mathbb{R}^{\prod_{n=1}^N d_n}$.
Form a \textup{TRP}(T) of order $N$ with range $k$
independent matrices whose entries are i.i.d. with
mean zero, variance one, and fourth moment $\Delta$.
Then
\begin{equation*}
\var(\|f_{\textup{TRP(T)}}(\mathbf{x})\|^2) = \frac{1}{Tk}(\Delta^N-3)\|\mathbf{x}\|_4^4 + \frac{2}{k}\|\mathbf{x}\|_2^4.
\end{equation*}
\end{lem}
We can see the variance increases with $N$.
In the $N=1$ Gaussian case, this formula shows a variance of $2/k \|\mathbf{x}\|_2^4$,
which agrees with the classic result.
Notice the TRP(T) only reduces the first term in the variance bound:
as $T\rightarrow \infty$, the variance converges to that of a Gaussian random map. Finally we show a non-asymptotic concentration bound for $N=2$.
We leave the parallel result for $N\ge 3$ open for future exploration.
\begin{thm}
	\label{prop: N-2-bound}
  Fix $\mathbf{x} \in \mathbb{R}^{d_1 d_2}$ with sub-Gaussian norm $\varphi_2$.
  Form a \textup{TRP}(T) of order 2 with range $k$
  composed of two independent matrices whose entries are drawn i.i.d.
  from a sub-Gaussian distribution with mean zero and variance one.
  Then there exists a constant $C$ depending on $\varphi_2$
  and a universal constant $c_1$ 
  so that
	\begin{equation}
	\mathbb{P}\left(\left| \|f_{\textup{TRP}}(\mathbf{x})\|^2 - \|\mathbf{x}\|^2_2 \right| \ge \epsilon \|x\|^2\right)\le C\exp\left[ - c_1 \left(\sqrt{k}\epsilon\right)^{1/2} \right],\nonumber
	\end{equation}

\end{thm}
Here $\varphi_2$ is the sub-Gaussian norm defined in \cref{def:sub-gaussian} in Appendix \ref{sec:appendix_proof}.
\cref{prop: N-2-bound} shows that for a TRP to form an $\epsilon$-JL DRM
with substantial probability on a dataset with $n$ points,
our method requires $k=\mathcal{O}(\epsilon^{-2}\log^4 n)$ while
conventional random projections require $k=\mathcal{O}(\epsilon^{-2}\log n)$.
Numerical experiments suggest this bound is pessimistic.

\section{Experiment} \label{sec:simulation}
In this section, we compare the quality of the isometry of
conventional RPs, TRP, and TRP(5), for
Gaussian, Sparse \cite{achlioptas2003database},
and Very Sparse random maps \cite{li2006very} on both synthetic data and MNIST data.
We also use TRP and TRP(5) to compute pairwise cosine similarity
(Table \ref{tbl:mnist_inner_prod} and Appendix \ref{appendix:more_result})
and to sketch matrices and tensors
(Appendix \ref{appendix:sketching}),
although the theory still remains open.

Our first experiment evaluates the quality of the isometry for maps $\mathbb{R}^d \to \mathbb{R}^k$.
We generate $n=10$ independent vectors $\mathbf{x}_1, \dots, \mathbf{x}_n$ of sizes
$d = 2500, 10000, 40000$ from $\mathcal{N}(\mathbf{0},\mathbf{I})$.
We consider the following three RPs:
1. Gaussian RP; 2. Sparse RP \cite{achlioptas2003database}; 3. Very Sparse RP \cite{li2006very}.
For each, we compare the performance of RP, TRP, and TRP(5) with order 2 and $d_1=d_2$.
We evaluate the methods by repeatedly generating a RP and computing the reduced vector,
and plot the ratio of the pairwise distance
$\frac{1}{n(n-1)}\sum_{n \geq i\neq j \geq 1}\frac{\|\mathbf{A}\mathbf{x}_i - \mathbf{A}\mathbf{x}_j\|_2}{\|\mathbf{x}_i - \mathbf{x}_j\|_2}$
and the average standard deviation for different $k$
averaged over 100 replications.
In the MNIST example, we choose the first $n = 50$ vectors of size $d = 784$, normalize them,
and perform the same experiment.
Figure \ref{fig:main} shows results on simulated ($d = 2500$) and MNIST data
for the Gaussian and Very Sparse RP.
See \cref{appendix:more_result} for additional experiments.

These experiments show that to preserve pairwise distance and cosine similarity,
TRP performs nearly as well as RP for all three types of maps.
With only five replicates, TRP(5) reduces the variance significantly in real data while not much in the simulation setting.
The difference in accuracy between methods diminishes as $k$ increases.
When $d = d_1d_2 = 40000$, the storage for TRP(5) is still $\frac{1}{20}$ of the Gaussian RP.
The variance reduction is effective especially in sparse and very sparse setting.
\begin{figure}[H]
	\centering
	\begin{subfigure}{0.24\textwidth}
		\includegraphics[scale = 0.22]{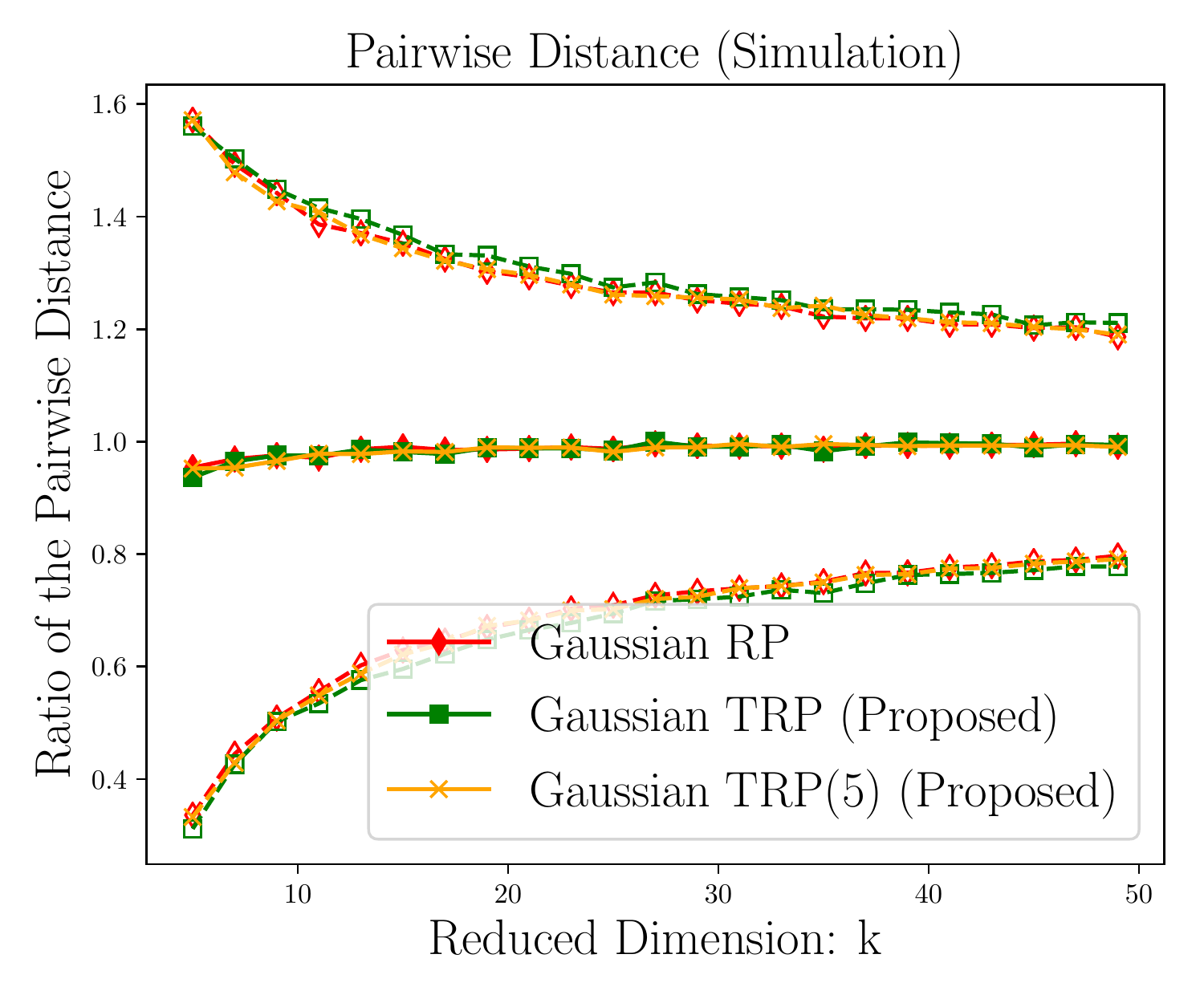}
	\end{subfigure}
	\begin{subfigure}{0.24\textwidth}
		\includegraphics[scale = 0.22]{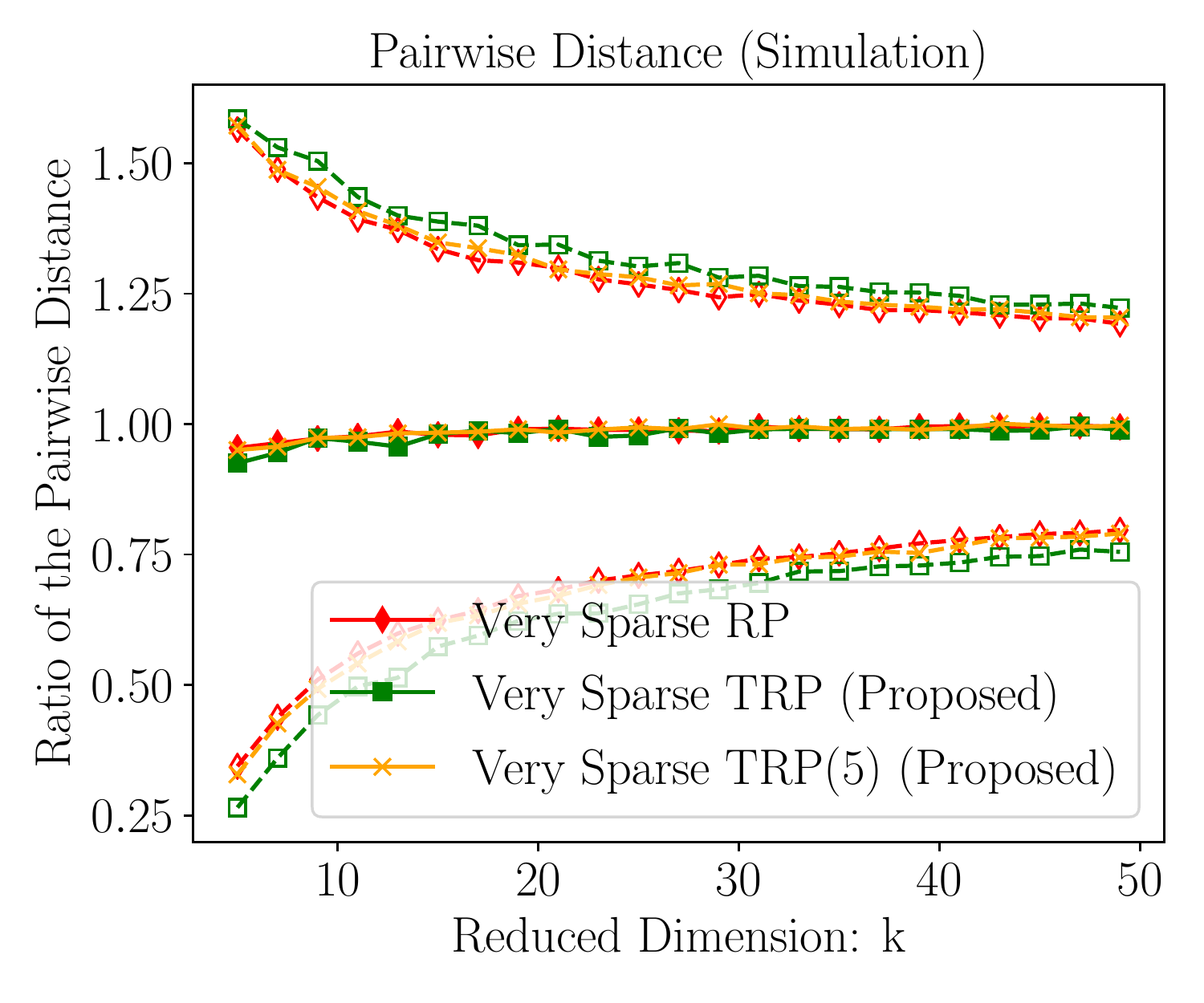}
	\end{subfigure}
	\begin{subfigure}{0.24\textwidth}
		\includegraphics[scale = 0.22]{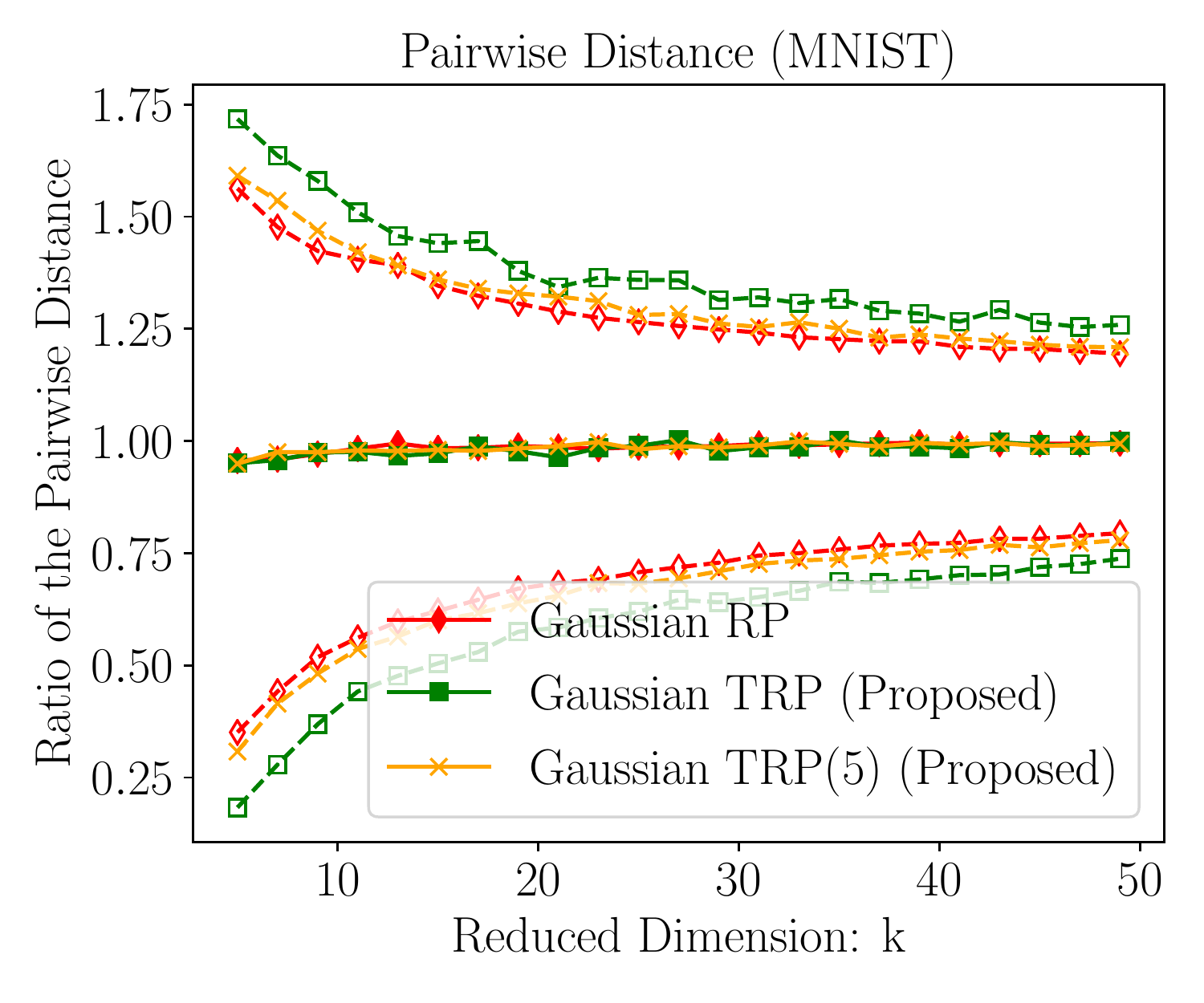}
	\end{subfigure}
	\begin{subfigure}{0.24\textwidth}
		\includegraphics[scale = 0.22]{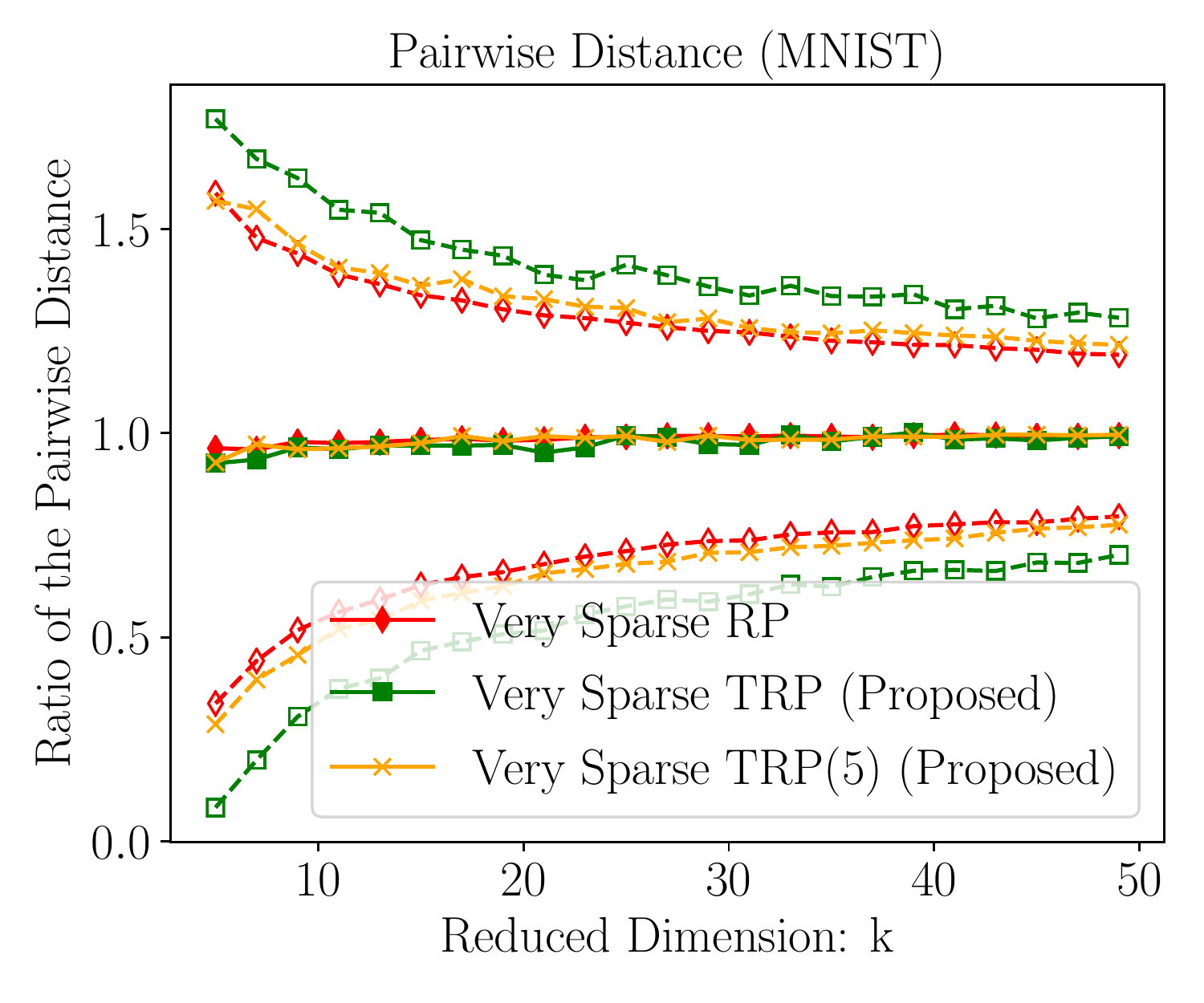}
	\end{subfigure}\\
	\caption{Isometry quality for simulated and MNIST data.
	The left two plots show results for Gaussian and Very Sparse RP, TRP, TRP(5)
	respectively applied to $n = 20$ standard normal data vectors in $\mathbb{R}^{2500}$.
	The right two plots show the same for 50 MNIST image vectors in $\mathbb{R}^{784}$.
	The dashed line shows the error two standard deviations from the average ratio.}
	\label{fig:main}
\end{figure}

\begin{table}[H]
\centering
\begin{tabular}{l|l|l|l}
       & Gaussian        & Sparse          & Very Sparse     \\ \hline
RP & 0.1198 (0.0147) & 0.1198 (0.0150) & 0.1189 (0.0108) \\ \hline
TRP    & 0.1540 (0.0290) & 0.1609 (0.0335) & 0.1662 (0.0307) \\ \hline
TRP(5) & 0.1262 (0.0166) & 0.1264 (0.0194) & 0.1276 (0.0164)
\end{tabular}
\caption{RMSE for the estimate of the pairwise inner product of the MNIST data, where standard error is in the parentheses.} \label{tbl:mnist_inner_prod}
\end{table}
\section{Conclusion}
The TRP is a novel dimension reduction map composed of smaller DRMs.
Compared to its constituent DRMs, it significantly reduces
the requirements for randomness and for storage.
Numerically, the variance-reduced TRP(5) method with only five replicates
achieves accuracy comparable to the conventional RPs for $1/20$ of the original storage.
We prove the TRP and TRP(T) are expected isometries with vanishing variance,
and provide a non-asymptotic error bound for the order 2 TRP.

For the future work, we will provide a general non-asymptotic bound
for the higher order TRP 
and develop the theory relevant for the application of the TRP
in sketching low-rank approximation, given its practical effectiveness
(shown in Appendix \ref{appendix:sketching}).


\clearpage
\bibliography{biblio}
\newpage
\begin{appendices}
\section{Proof for Main Theorems}
\label{sec:appendix_proof}

Before presenting the proof of the main theory, we introduce some additional notation.

\subsection*{Notations for Technical Proofs}
For a vector $\mathbf{x}$ with length $\prod_{n=1}^N d_n$,
we introduce the multi-index $\mathbf{x}_{r_1, \cdots, r_N}$, for any $r_n \in [d_n]$,
to represent the $(1 + \sum_{n = 1}^N(r_n -1)s_n)^{th}$ element,
where $s_n = \prod_{n+1}^Nd_n$ for $n < N$ and $s_n = 1$ for $n = N$.
For a vector $\mathbf{r}_1$, $\mathbf{r}_2$, we say $\mathbf{r}_1 = \mathbf{r}_2$
if and only if all their elements are the same.

Also, we let $\mathop{\mathbf{vec}}(\mathbf{A})$ be the vectorization operator for any matrix $\mathbf{A}\in \mathbb{R}^{d\times k}$, which stacks all columns of matrix $\mathbf{A}$ and returns a vector of length $kd$,
$[\mathbf{A}(\cdot, 1); \cdots; \mathbf{A}(\cdot, k); ]$. Here we use the semi-colon to denote
stacking vectors $\mathbf{x}$ and $\mathbf{y}$ on top of each other (vertically) to form a longer vector as $[\mathbf{x}; \mathbf{y}]$.
As comparison, we use the comma to denote row-wise concatenation, stacking vectors horizontally to form a matrix with two columns as $[\mathbf{x}^\top, \mathbf{y}^\top]$.

\subsection*{Proof for Lemma \ref{lemma: norm-preserve}}
\begin{proof}

We first give sufficient conditions on a random matrix that guarantees the conclusion of Lemma \ref{lemma: norm-preserve} holds
when the map is simply multiplication by the random matrix.
Then we show that the Khatri-Rao map with the condition stated in Lemma \ref{lemma: norm-preserve}
satisfies these sufficient conditions.

Consider a random matrix $\mathbf{A}\in \mathbb{R}^{k \times d}$ and $\mathbf{x} \in \mathbb{R}^d$
whose entries have unit variance and for which entries in the same row are uncorrelated:
$\mathbb{E}{\mathbf{A}^2(r,s)}=1$ for all $r$ and $s$ and $\mathbb{E}{\mathbf{A}(r,s_1)\mathbf{A}(r,s_2)} = 0$
for all $r \in[k]$, $s_1\neq s_2 \in[d]$.
Then $\mathbb{E}\| \frac{1}{\sqrt{k}}\mathbf{y}\|_2^2 = \|\mathbf{x}\|_2^2$, when $\mathbf{y}=\mathbf{A}\mathbf{x}$.

To see why, it suffices to show that $\mathbb{E} y_r^2 = \|x\|^2_2$:
\begin{equation} \label{eq:row-length}
\begin{aligned}
&\mathbb{E}{y_r^2} = \mathbb{E}{\sum_{s_1=1}^d\sum_{s_2=1}^d \mathbf{A}(r,s_1)\mathbf{A}(r,s_2)x_{s_1}x_{s_2}} \\
&= \sum_{s=1}^d \mathbf{A}^2(r,s)x_s^2 = \|\mathbf{x}\|^2_2,  \nonumber
\end{aligned}
\end{equation}
where the first equation on the second line comes from the fact that $\mathbb{E} {\mathbf{A}(r,s_1)\mathbf{A}(r,s_2)}  = 0$ for $s_1\neq s_2$ and the second equation on the second line uses $\mathbb{E}{\mathbf{A}^2(r,s)} =1$.

Now we prove Lemma \ref{lemma: norm-preserve} by induction.
We first show that for two matrices $\mathbf{B}_1 \in \mathbb{R}^{d_1\times k}, \mathbf{B}_2 \in \mathbb{R}^{d_2\times k}$
whose entries have unit variance and are uncorrelated within the same row,
the Khatri-Rao product $\mathbf{A} = (\mathbf{B}_1 \odot\mathbf{B}_2 )^\top $
also has entries with unit variance that are uncorrelated within the same row.

For the proof, it suffices to restrict our focus to the first row of $\mathbf{\Omega}$.
For any $1\le r_1\le d_1, 1\le r_2\le d_2$,
\begin{equation}
\begin{aligned}
&\mathbb{E}  \mathbf{A}^2_{1}(k_1,k_2) = \mathbb{E} \mathbf{B}^2_{1}(k_1,1)\mathbf{B}^2_{2}(k_2,1)\\
& =  \mathbb{E} \mathbf{B}^2_{1}(k_1,1) \mathbb{E} \mathbf{B}^2_{2}(k_1,1) = 1,
\nonumber
\end{aligned}
\end{equation}
using the independence between $\mathbf{B}_1$ and $\mathbf{B}_2$.
To avoid confusion in notation, we argue that $ \mathbf{A}(1,\cdot)$ is the first row vector of $ \mathbf{A}$ of size $d_1d_2$, and we apply the multi-index to it.  Also, for two different elements in the first row of $\mathbf{A}$: $ \mathbf{A}_{1}(k_1,k_2) \mathbf{A}_{1}(s_1,s_2)$ at least one of $k_1\neq s_1$, $k_2\neq s_2$ hold.
Without loss of generality, assuming $k_1\neq s_1$,
\begin{equation}
\begin{aligned}
&\mathbb{E}  \mathbf{A}_{1}(k_1,k_2) \mathbf{A}_{1}(s_1,s_2)  = \mathbb{E} \mathbf{B}_{1}(k_1,1)\mathbf{B}_{2}(k_2,1)\mathbf{B}_{1}(s_1,1)\mathbf{B}_{2}(s_2,1)\\
& =  \mathbb{E} \mathbb{E}\left[ \mathbf{B}_{1}(k_1,1)\mathbf{B}_{1}(k_2,1)\mathbf{B}_{2}(k_2,1)\mathbf{B}_{2}(s_2,1)  \mid  \mathbf{B}_{2}(k_2,1)\mathbf{B}_{2}(s_2,1)\right]\\
& =  \mathbb{E}  \mathbf{B}_{2}(k_2,1)\mathbf{B}_{2}(s_2,1) \mathbb{E}\left[ \mathbf{B}_{1}(k_1,1)\mathbf{B}_{1}(s_1,1) \right]  = 0,
\nonumber
\end{aligned}
\end{equation}
where we use the fact that entries in the same row for each constituent matrix $B_1$ and $B_2$ are mutually uncorrelated.

Noting that the two important properties --- unit variance and zero correlation within rows ---
are preserved by the Khatri-Rao product,
we can finish the proof of the lemma, for the case of TRP(1), with a standard mathematical induction argument.
For TRP(T),
\begin{equation}
\begin{aligned}
&\mathbb{E}\|f_{\text{TRP(T)}}(\mathbf{x})\|^2_2=\frac{1}{T} \mathbb{E}\|\sum_{t=1}^T f^{(t)}_{\textup{TRP}}(\mathbf{x})\|^2_2\\
&= \frac{1}{T}\sum_{t=1}^T\mathbb{E} \|f^{(t)}_{\textup{TRP}}(\mathbf{x})\|^2_2= \|x\|^2_2,
\nonumber
\end{aligned}
\end{equation}
where in the second line we use the fact that each $f^{(t)}_{\textup{TRP}}$ is independent with each other.
\end{proof}

Next we introduce a lemma that shows we can control the deviation of the inner product after applying a map
by controlling the deviation of the square norm.
This result is well known in the literature on random projections.

\begin{lem}
\label{lemma:inner-product}
For a linear mapping from $\mathbb{R}^d\rightarrow \mathbb{R}^k$: $f(\mathbf{x}) = \frac{1}{\sqrt{k}}\mathbf{\Omega x}$,
\begin{equation}
\label{eq:inner-bound}
\mathbb{P}(|\langle f(\mathbf{x}), f(\mathbf{y})\rangle - \langle \mathbf{x}, \mathbf{y}\rangle|\ge \epsilon |\langle \mathbf{x}, \mathbf{y}\rangle|) \le 2\sup_{\mathbf{x}\in \mathbb{R}^{d}}\mathbb{P}(| \|f(\mathbf{x})\|^2-\|\mathbf{x}\|^2|\ge \epsilon \|\mathbf{x}\|_2^2).\nonumber
\end{equation}
\end{lem}

\begin{proof}
Since $f$ is a linear mapping, we have
\[
4f(\mathbf{x})f(\mathbf{y}) = \|f(\mathbf{x}+\mathbf{y})\|^2_2-  \|f(\mathbf{x}-\mathbf{y})\|^2_2.
\]
Consider the event

\begin{equation}
\begin{aligned}
 &\mathcal{A} _1=  \left\{ |\|f(\mathbf{x}+\mathbf{y})\|^2_2-\|\mathbf{x}+\mathbf{y}\|^2_2|\ge \epsilon \|\mathbf{x}+\mathbf{y}\|_2^2 \right\}\\
 &\mathcal{A} _2=  \left\{ | \|f(\mathbf{x}-\mathbf{y})\|^2_2-\|\mathbf{x}-\mathbf{y}\|^2_2|\ge \epsilon \|\mathbf{x}-\mathbf{y}\|_2^2 \right\} \nonumber
 \end{aligned}
 \end{equation}

 On the event $\mathcal{A}_1^\complement \cap\mathcal{A}_2^\complement$,
 \[
 4f(\mathbf{x})f(\mathbf{y})  \ge (1-\epsilon) (\mathbf{x}+\mathbf{y})^2 - (1+\epsilon) (\mathbf{x}-\mathbf{y})^2 = 4 \langle \mathbf{x}, \mathbf{y}\rangle -2\epsilon (\|\mathbf{x}\|^2+\|\mathbf{y}\|^2),
 \]
 noticing $\|\mathbf{x}\|^2+\|\mathbf{y}\|^2\ge 2\langle \mathbf{x},  \mathbf{y}\rangle$, and by similar argument on the other side of the inequality, we could claim that
\[
\left\{|\langle f(\mathbf{x}), f(\mathbf{y})\rangle - \langle \mathbf{x}, \mathbf{y}\rangle|\ge \epsilon |\langle \mathbf{x}, \mathbf{y}\rangle|\right\} \subseteq  \mathcal{A}_1 \cup \mathcal{A}_2.
\]
We finish the proof by simply applying the union bound of the two events.
\end{proof}
\begin{remark}
One key element of classic random projection results is the dimension-free bound.
According to Prop. \ref{prop: N-2-bound}, our TRP has a norm preservation bound independent of the particular vector  $\mathbf{x}$ and dimension $d$ and thus a dimension-free inner product preservation bound according to Lemma \ref{eq:inner-bound}.
\end{remark}

\subsection*{Proof for Lemma \ref{lemma:variance}}
\begin{proof}
Let $\mathbf{y}=\mathbf{Ax}$. We know from Lemma \ref{lemma: norm-preserve} that $\mathbb{E}\|f_{\textup{TRP}}(\mathbf{x})\|^2_2 = \frac{1}{k}\mathbb{E}\|\mathbf{Ax}\|^2= \|\mathbf{x}\|^2_2$.
Notice
\[
\mathbb{E}(\|f_{\text{TRP(T)}}(\mathbf{x})\|^2_2) = \|\mathbf{x}\|^2_2,
\]
and $\mathbb{E} y_1^2=\|x\|_2^2$ as shown in the poof of Lemma \ref{lemma: norm-preserve}. It is easy to see that
\[
\mathbb{E}\|\mathbf{y}\|^4_2 = \sum_{i=1}^k \mathbb{E} y_i^4 + \sum_{i\neq j} \mathbb{E} y_i^2y_j^2.
\]
Again, as shown in Lemma \ref{lemma: norm-preserve}, $\mathbb{E} y_i^2y_j^2 =\mathbb{E} y_i^2 \mathbb{E} y_j^2 = \|\mathbf{x}\|^4$. To find $\mathbb{E}\|\mathbf{y}\|^4_2$, it suffices to find $\mathbb{E} y_1^4$ by noticing that $y_i$ are i.i.d. random variables. Let $\Omega$ be the set containing all corresponding multi-index vector for $\{1,\cdots, \prod_{n=1}^N d_n\}$.
\begin{equation}
\begin{aligned}
&y_1^4 = \left[\sum_{\mathbf{r}\in \Omega} \mathbf{A}(1,\mathbf{r}) x_{\mathbf{r}}\right]^4 = \sum_{\mathbf{r}\in \Omega} \mathbf{A}^4(1,\mathbf{r}) x^4_{\mathbf{r}} + 3\sum_{\mathbf{r}_1 \neq \mathbf{r}_2 \in \Omega} \mathbf{A}^2(1,\mathbf{r}_1) x^2_{\mathbf{r}_1}\mathbf{A}^2(1,\mathbf{r}_2) x^2_{\mathbf{r}_2}\\
&+6\sum_{\mathbf{r}_1 \neq \mathbf{r}_2 \neq \mathbf{r}_3 \in \Omega} \mathbf{A}^2(1,\mathbf{r}_1) x_\mathbf{\mathbf{r}_1} \mathbf{A}(2,\mathbf{r}_2)x_{\mathbf{r}_2}\mathbf{A}(3,\mathbf{r}_3)x_{\mathbf{r_3}}+4\sum_{\mathbf{r}_1 \neq \mathbf{r}_2 \in \Omega} \mathbf{A}^3(1,\mathbf{r}_1) x^3_{\mathbf{r}_1}\mathbf{A}(1,\mathbf{r}_2) x_{\mathbf{r}_2}\\
&+\sum_{\mathbf{r}_1 \neq \mathbf{r}_2 \neq \mathbf{r}_3\neq \mathbf{r}_4 \in \Omega} \mathbf{A}(1,\mathbf{r}_1) x_{\mathbf{r}_1}\mathbf{A}(1,\mathbf{r}_2) x_{\mathbf{r}_2}\mathbf{A}(1,\mathbf{r}_3) x_{\mathbf{r}_3}\mathbf{A}(1,\mathbf{r}_4) x_{\mathbf{r}_4}.\nonumber
\end{aligned}
\end{equation}
Recall that elements of $A$ are uncorrelated within rows, as shown in the proof of in Lemma \ref{lemma: norm-preserve}.
Hence the expectation of the terms on the second and third lines is zero.
Again using uncorrelatedness within rows,
\[
\mathbb{E} \mathbf{A}^2(1,\mathbf{r}_1)\mathbf{A}^2(1,\mathbf{r}_2) = \mathbb{E} \mathbf{A}^2(1,\mathbf{r}_1) \mathbb{E}  \mathbf{A}^2(1,\mathbf{r}_2)  = 1.
\]
Each element of $A$ has fourth moment $\Delta$, so
\[
\mathbb{E} \mathbf{A}^4(1,\mathbf{r}) = \mathbb{E} \mathbf{A}^4_1 (1, r_1) \cdots \mathbf{A}^4_N(1, r_N) = \Delta^N.
\]
Combining these two together, we have
\begin{equation}
\begin{aligned}
&\mathbb{E}\|f_{\textup{TRP}}(\mathbf{x})\|^4 = \frac{1}{k^2}\left[k(\Delta^N-3)\|\mathbf{x}\|_4^4 +3k\|\mathbf{x}\|_2^4  +(k-1)k \|\mathbf{x}\|_2^4\right]\\
& =  \frac{1}{k}\left[(\Delta^N-3)\|\mathbf{x}\|_4^4 +2\|\mathbf{x}\|_2^4\right]+\|\mathbf{x}\|_2^4.  \nonumber
\end{aligned}
\end{equation}
Therefore,
\[
\textrm{Var}(\|f_{\textup{TRP}}(\mathbf{x})\|^2_2) = \mathbb{E}\|f_{\textup{TRP}}(\mathbf{x})\|^4_2 -  (\mathbb{E}\|f_{\textup{TRP}}(\mathbf{x})\|^2_2)^2 = \frac{1}{k}\left[(\Delta^N-3)\|\mathbf{x}\|_4^4 +2\|\mathbf{x}\|_2^4\right].
\]
Now we switch to see how much variance could be reduced by the variance reduction method. With Lemma \ref{lemma: norm-preserve}, we already know that $\mathbb{E} \|f_{\text{TRP(T)}}(\mathbf{x})\|^2_2= \|\mathbf{x}\|_2^2$. The rest is to calculate $\mathbb{E} \|f_{\text{TRP(T)}}(\mathbf{x})\|^4_2$ out.
\begin{equation}
\begin{aligned}
&\|f_{\text{TRP(T)}}(\mathbf{x})\|^4_2 = \frac{1}{T^2}\left[\sum_{t=1}^T \|f^{(t)}_{\textup{TRP}}(\mathbf{x})\|^2_2+\sum_{t_1 \neq t_2} \langle f^{(t_1)}_{\textup{TRP}}(\mathbf{x}), f^{(t_2)}_{\textup{TRP}}(\mathbf{x}) \rangle \right]^2\\
&= \frac{1}{T^2} \left[ \sum_{t=1}^T \|f^{(t)}_{\textup{TRP}}(\mathbf{x})\|^4_2+\sum_{t_1\neq t_2}\|f_{\textup{TRP}}^{(t_1)}(\mathbf{x})\|^2_2\|f_{\textup{TRP}}^{(t_2)}(\mathbf{x})\|^2_2 +2\sum_{t_1\neq t_2}\langle f^{(t_1)}_{\textup{TRP}}(\mathbf{x}), f^{(t_2)}_{\textup{TRP}}(\mathbf{x}) \rangle^2 +\text{rest} \right].\nonumber
\end{aligned}
\end{equation}
It is not hard to show that $\mathbb{E}(\text{rest}) = 0$. Following the definition of $\mathbf{y}$,
\begin{equation}
\mathbb{E} \|f_{\textup{TRP}}^{(t_1)}(\mathbf{x})\|^2_2\|f_{\textup{TRP}}^{(t_2)}(\mathbf{x})\|^2_2= \|\mathbf{x}\|_2^4 \nonumber
\end{equation}
and
\begin{equation}
\begin{aligned}
&\mathbb{E} \langle f^{(t_1)}_{\textup{TRP}}(\mathbf{x}), f^{(t_2)}_{\textup{TRP}}(\mathbf{x}) \rangle^2\\ &=\frac{1}{k^2} \mathbb{E} \left[ \sum_{i=1}^k y^{(t_1)}_i y^{(t_2)}_i \right]^2 = \frac{1}{k}  \mathbb{E} [y^{(t_1)}_1 y^{(t_2)}_1]^2 = \frac{1}{k}\|x\|_2^4. \nonumber
\end{aligned}
\end{equation}
Combining all these together, we see that
\begin{equation}
\begin{aligned}
\textrm{Var}(\|f_{\text{TRP(T)}}(\mathbf{x})\|_2^2) &= \mathbb{E}\|f_{\text{TRP(T)}}(\mathbf{x})\|^4_2 -  (\mathbb{E}\|f_{\text{TRP(T)}}(\mathbf{x})\|^2_2)^2\\
&= \frac{1}{T^2} \left[ \frac{T}{k}\left[(\Delta^N-3)\|\mathbf{x}\|_4^4 +2\|\mathbf{x}\|_2^4 \right]\right.\\
&\left. +T(T-1)\|\mathbf{x}\|_2^4+T\|\mathbf{x}\|_2^4+\frac{2T(T-1)}{k}\|\mathbf{x}\|_2^4 \right] - \|\mathbf{x}\|_2^4 \\
&= \frac{1}{Tk}(\Delta^N-3)\|\mathbf{x}\|_4^4 + \frac{2}{k}\|\mathbf{x}\|_2^4. \nonumber
\end{aligned}
\end{equation}
\end{proof}

The proof of our main theorem requires an additional definition.

\begin{definition}
	\label{def:generalized-sub-exponential-mc}
	A random variable $x$ is said to satisfy the generalized-sub-exponential moment condition with constant $\alpha$, if for general positive integer $k$, there exists a general constant $C$(not depending on k), so that
	\begin{equation}
	\mathbb{E} |x|^k \le (Ck)^{k \alpha}.
	\end{equation}
\end{definition}

Before we present our proof, we state a simplification of Lemma B.2 from \cite{erdHos2012bulk}
that we will use to prove the main theorem.
For details of the proof, we refer readers to \cite{erdHos2012bulk}.

\begin{lem}
\label{lemma:sub-exponential-deviation}
Suppose the i.i.d. random variables $x_i$ have mean 0 and variance 1,
and satisfy Definition \ref{def:generalized-sub-exponential-mc} with parameter $\alpha$.
Then
\begin{equation}
\mathbb{P}\left(\left|\frac 1 n \sum_{i=1}^n x_i^2 - 1 \right| > Dn^{-1/2}\right) \le C \exp(-c D^{\frac{1}{1+\alpha}})
\end{equation}
for some $c$ depending on $\alpha$.
\end{lem}

\subsection*{Proof for Theorem \ref{prop: N-2-bound}}
\begin{proof}
From now on, without loss of generality, we will assume $\|x\|=1$.  Let
\begin{equation}
\mathbf{y}= \frac{1}{\sqrt{k}}(\mathbf{A}_1 \odot \mathbf{A}_2)^\top \mathbf{x}, \nonumber
\end{equation}
Lemma \ref{lemma: norm-preserve} asserts that $\mathbb{E}\|\mathbf{y}\|^2_2= \|\mathbf{x}\|^2_2$,
as the conditions required for Lemma \ref{lemma: norm-preserve} hold for i.i.d. random variables with mean 0 and variance 1.
The key observation is that for each $i\in [k]$, $y_i$ is some quadratic function of
the $i$th column of the matrices $\mathbf{A}_1$ and $\mathbf{A}_2$.
As a quadratic function of \emph{mutually independent, mean 0, sub-Gaussian} variables, each $y_i$
satisfies a concentration bound: for example,
the Hanson-Wright inequality (Lemma \ref{lemma:hanson_wright}).

Let's be explicit.
We aim to write $y_i$ as a quadratic form of $\mathbf{z}_i:=[\mathop{\mathbf{vec}}(\mathbf{A}_{1}(\cdot,i)); \mathop{\mathbf{vec}}(\mathbf{A}_{2}(\cdot,i))]$. Also, for convenience, we partition $\mathbf{x}$ into $d_1$ sub-vectors with equal length $d_2$ i.e., $\mathbf{x} = [\mathbf{x}_1; \cdots; \mathbf{x}_{d_1}]$.
To be clear, we write $y_1$ as a quadratic form of $\mathbf{z}_1$ first.
\begin{equation}
y_1 = \langle [\mathbf{A}_{1}(1,1) \mathbf{A}_{2}(\cdot,1); \cdots; \mathbf{A}_{1}(d_1,1) \mathbf{A}_{2}(\cdot,1)], [\mathbf{x}_1;\cdots; \mathbf{x}_{d_1}]\rangle
\nonumber
\end{equation}
which indicates that we could write
\begin{equation}
y_1 = \mathbf{z}^\top_1
\mathbf{M}\mathbf{z}_1, \nonumber
\end{equation}
where
\begin{equation}
\mathbf{M} = \begin{bmatrix}
\mathbf{0} & \mathbf{D}\\
\mathbf{0} & \mathbf{0}
\end{bmatrix}  ~~
\mathbf{D} = \begin{bmatrix}
\mathbf{x}_1^\top \\
\vdots \\
\mathbf{x}_{d_1}^\top.  \nonumber
\end{bmatrix}
\end{equation}
It is easy to see that $\|\mathbf{M}\| \le \|\mathbf{D}\| \le \|\mathbf{D}\|_F = \|\mathbf{M}\|_F = 1$ by assuming $\|\mathbf{x}\| = 1$.

Now applying the Hanson Wright inequality (Lemma \ref{lemma:hanson_wright}),
we see that for any positive number $\eta$, there exists a general constant $c_1$ so that
\begin{equation}
\begin{aligned}
 \mathbb{P}(|y_i|\ge \eta) & \le 2\exp\left[-c_1\min\left\{ -\frac{\eta}{\varphi_2^2 \|M\|} ,  \frac{\eta^2}{\varphi_2^4 \|M\|^2_F} \right\} \right] \\
& \le  2\exp\left[-c_1\min\left\{ -\frac{\eta}{\varphi_2^2} ,  \frac{\eta^2}{\varphi_2^4 } \right\} \right].  \nonumber
\end{aligned}
\end{equation}

Next, using Lemma \ref{lemma:hanson-wright-sub-exponential}, we assert that there is a constant $C$
depending on the constant $c_1$ and the sub-Gaussian norm $\varphi_2$ so that
\begin{equation}
\mathbb{E} |y_i|^k \le (Ck)^k. \nonumber
\end{equation}
In fact, we can use the lemma to find an explicit constant $C$ that suffices:
\begin{equation}\label{eq:constant_in-sub-exponential}
C = 1+ \frac{c_1}{\min\left\{\varphi_2^2, \varphi_2^4\right\}}.
\end{equation}

Finally, notice the $y_i$ satisfy the assumptions of Lemma \ref{lemma:sub-exponential-deviation}.
In particular, the columns of $A_1 \odot A_2$ are iid,
and so the $y_i$ are iid.
We use the lemma to see that
\begin{equation}
\begin{aligned}
& \mathbb{P}\left(\left|\frac{1}{k} \mathbf{y}^\top \mathbf{I}_{k,k} \mathbf{y}-1 \right|\ge \epsilon\right )  \le C\exp\left( - c_2 \left[\sqrt{k}\epsilon\right]^{1/2} \right),
\nonumber
\end{aligned}
\end{equation}
where $C$ is defined in \eqref{eq:constant_in-sub-exponential} for some constant $c_2$ with $\alpha$ set to be 1.

\end{proof}

\clearpage

\section{More Simulation Results}\label{appendix:more_result}

\paragraph{Pairwise Distance Estimation}
In Figure \ref{fig:gaussian}, \ref{fig:sparse}, \ref{fig:very_sparse}, we compare the performance of Gaussian, Sparse, Very Sparse random maps on the pairwise distance estimation problem with $d = 2500, 10000, 40000, N= 2$. Additionally, we compare their performance for $d = 125000, N = 3$ in Figure \ref{fig:triple_krao}.

\begin{figure}[ht!]
	\centering
	\begin{subfigure}{0.32\textwidth}
		\includegraphics[scale = 0.29]{figure/dist_g_d2500.pdf}
	\end{subfigure}
	\begin{subfigure}{0.32\textwidth}
		\includegraphics[scale = 0.29]{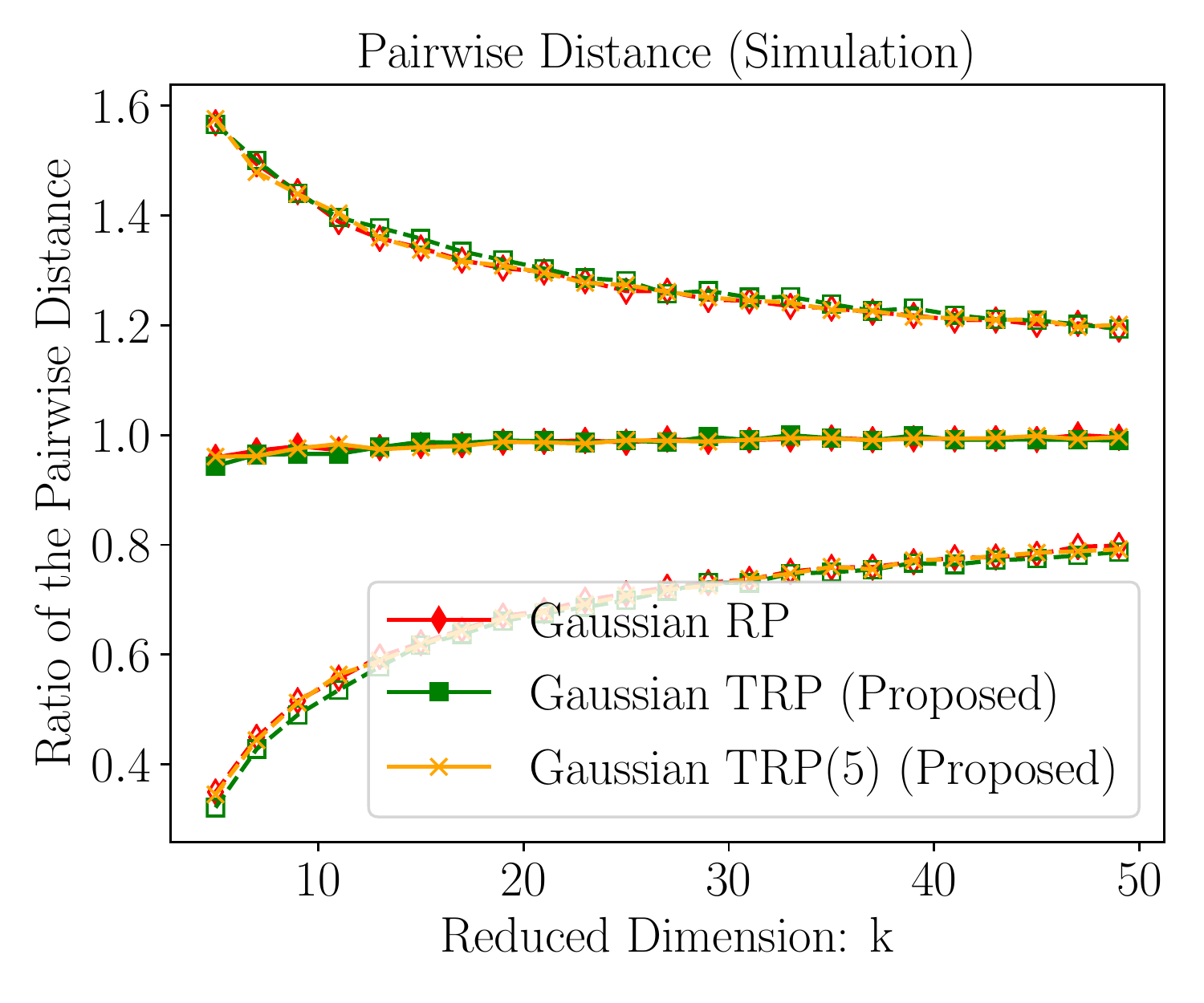}
	\end{subfigure}
	\begin{subfigure}{0.32\textwidth}
		\includegraphics[scale = 0.29]{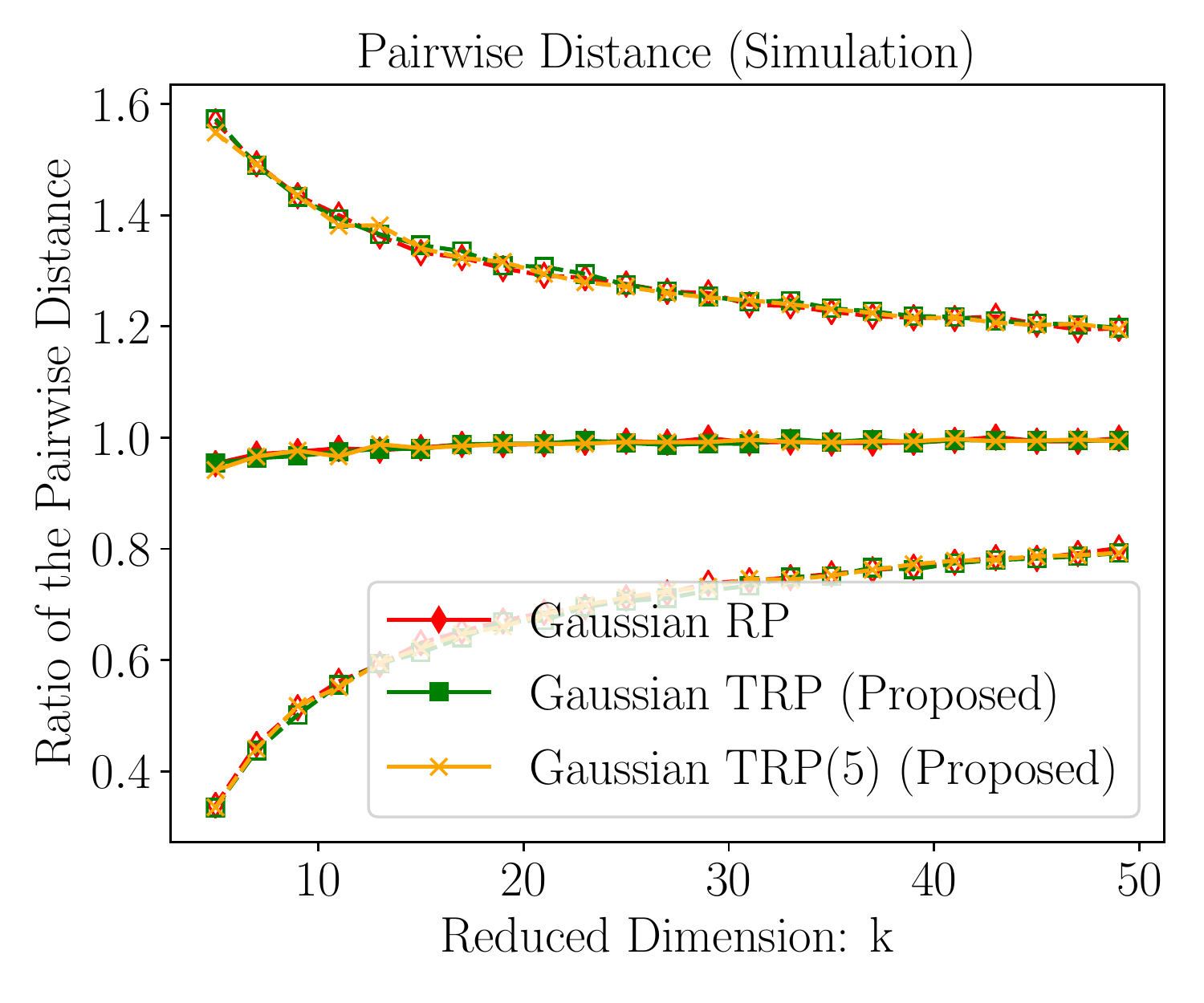}
	\end{subfigure}
	\\
	\caption{Average ratio of the pairwise distance for simulation data using Gaussian RP: \textit{The  plots correspond to the simulation for Gaussian RP, TRP, TRP(5) respectively with $n = 20, d = 2500, 10000, 40000$ and each data vector comes from $N(\mathbf{0}, \mathbf{I})$. The dashed line represents the error bar 2 standard deviation away from the average ratio.}} 
	\label{fig:gaussian}
\end{figure}

\begin{figure*}[ht!]
	\centering
	\begin{subfigure}{0.32\textwidth}
		\includegraphics[scale = 0.29]{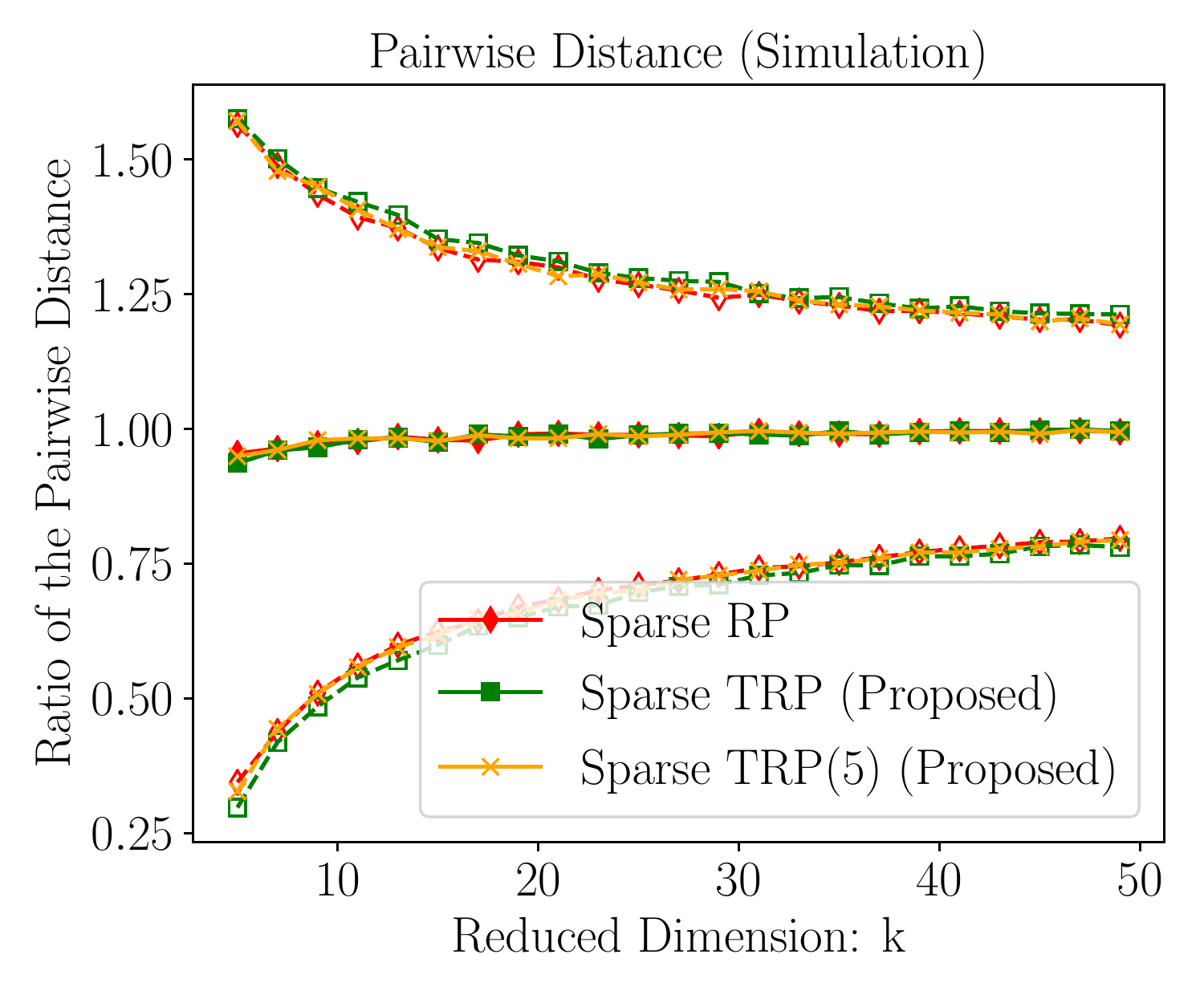}
	\end{subfigure}
	\begin{subfigure}{0.32\textwidth}
		\includegraphics[scale = 0.29]{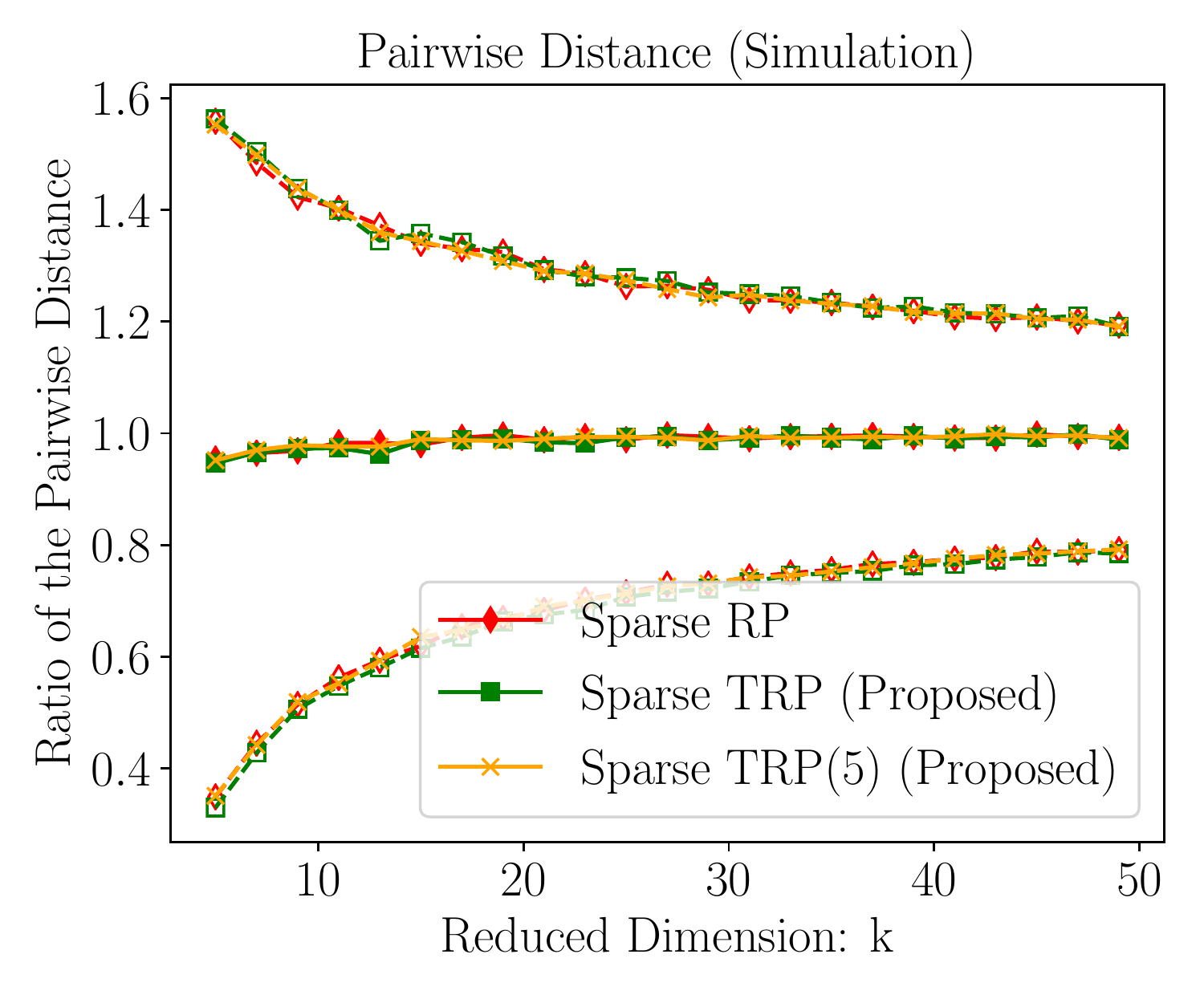}
	\end{subfigure}
	\begin{subfigure}{0.32\textwidth}
		\includegraphics[scale = 0.29]{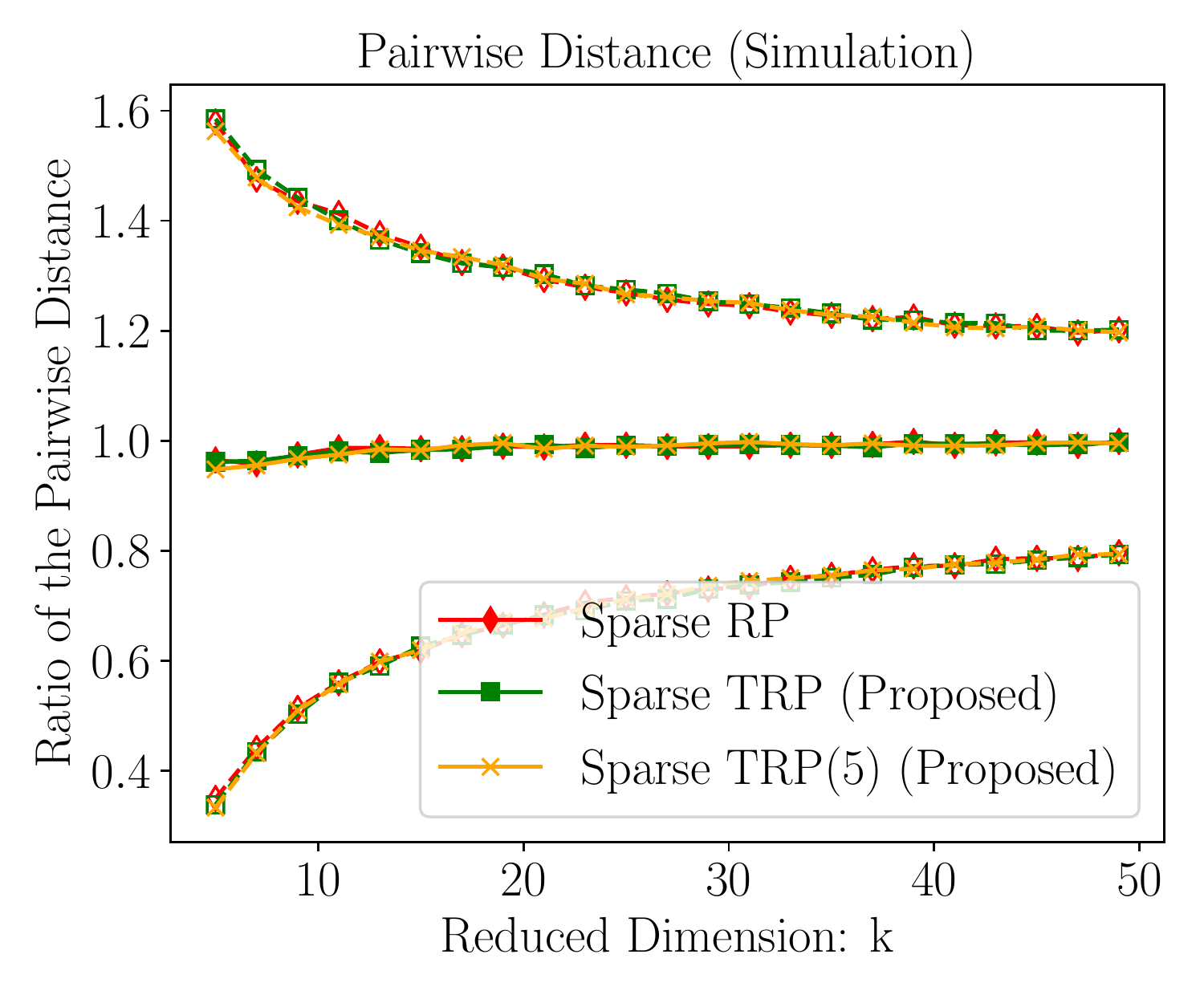}
	\end{subfigure}\\
	\caption{Average ratio of the pairwise distance for simulation data using Sparse RP: \textit{The  plots correspond to the simulation for Sparse RP, TRP, TRP(5) respectively with $n = 20, d = 2500, 10000, 40000$ and each data vector comes from $N(\mathbf{0}, \mathbf{I})$. The dashed line represents the error bar 2 standard deviation away from the average ratio.}} 
	\label{fig:sparse}
\end{figure*}

\begin{figure*}[ht!] 
	\centering
	\begin{subfigure}{0.32\textwidth}
		\includegraphics[scale = 0.29]{figure/dist_sp1_d2500.pdf}
	\end{subfigure}
	\begin{subfigure}{0.32\textwidth}
		\includegraphics[scale = 0.29]{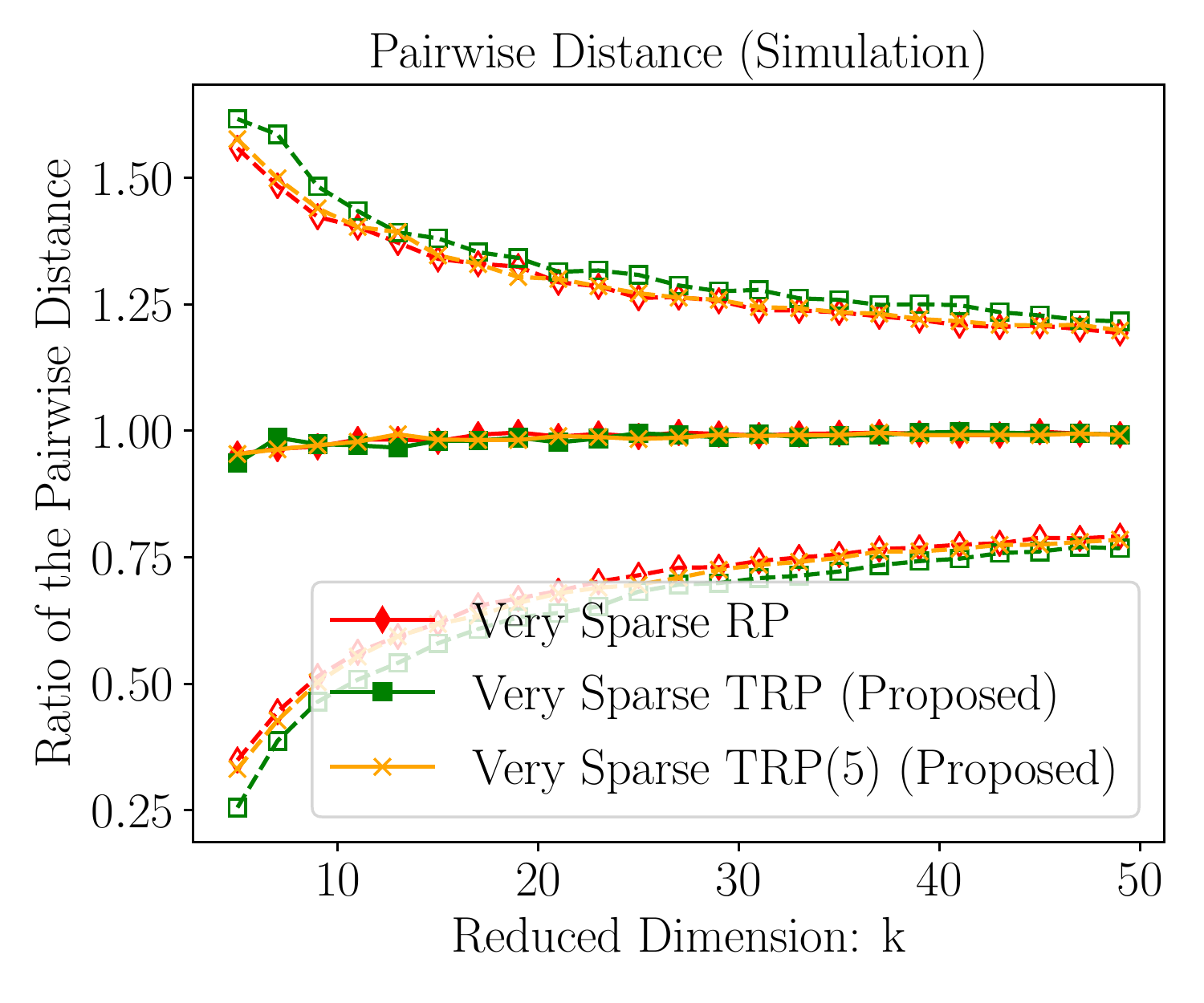}
	\end{subfigure}
	\begin{subfigure}{0.32\textwidth}
		\includegraphics[scale = 0.29]{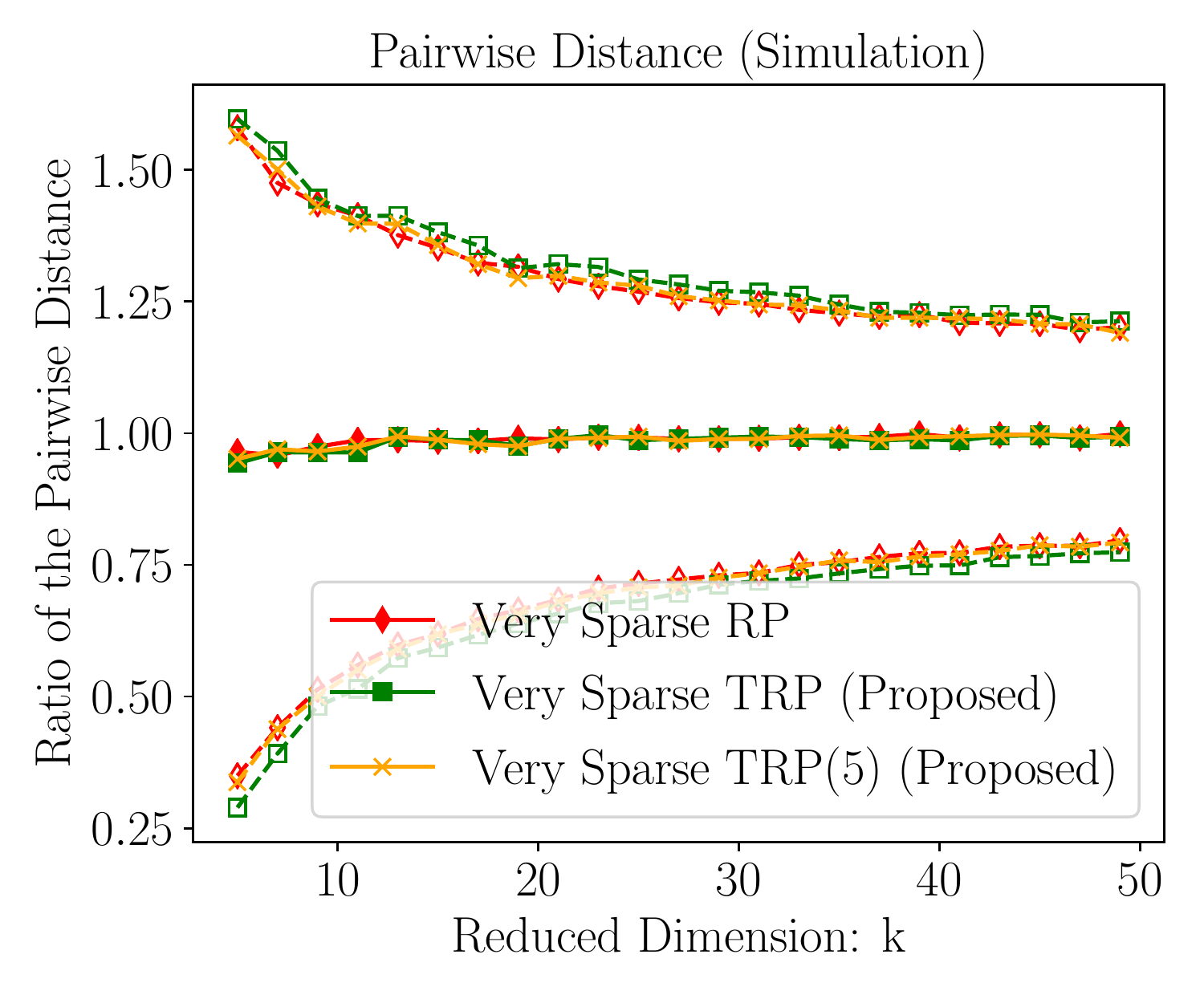}
	\end{subfigure}\\
	\caption{Average ratio of the pairwise distance for simulation data using Very Sparse RP: \textit{The plots correspond to the simulation for Very Sparse RP, TRP, TRP(5) respectively with $n = 20, d = 2500, 10000, 40000$ and each data vector comes from $N(\mathbf{0}, \mathbf{I})$. The dashed line represents the error bar 2 standard deviation away from the average ratio.}} 
	\label{fig:very_sparse}
\end{figure*}

\begin{figure*}[ht!] 
	\centering
	\begin{subfigure}{0.32\textwidth}
		\includegraphics[scale = 0.29]{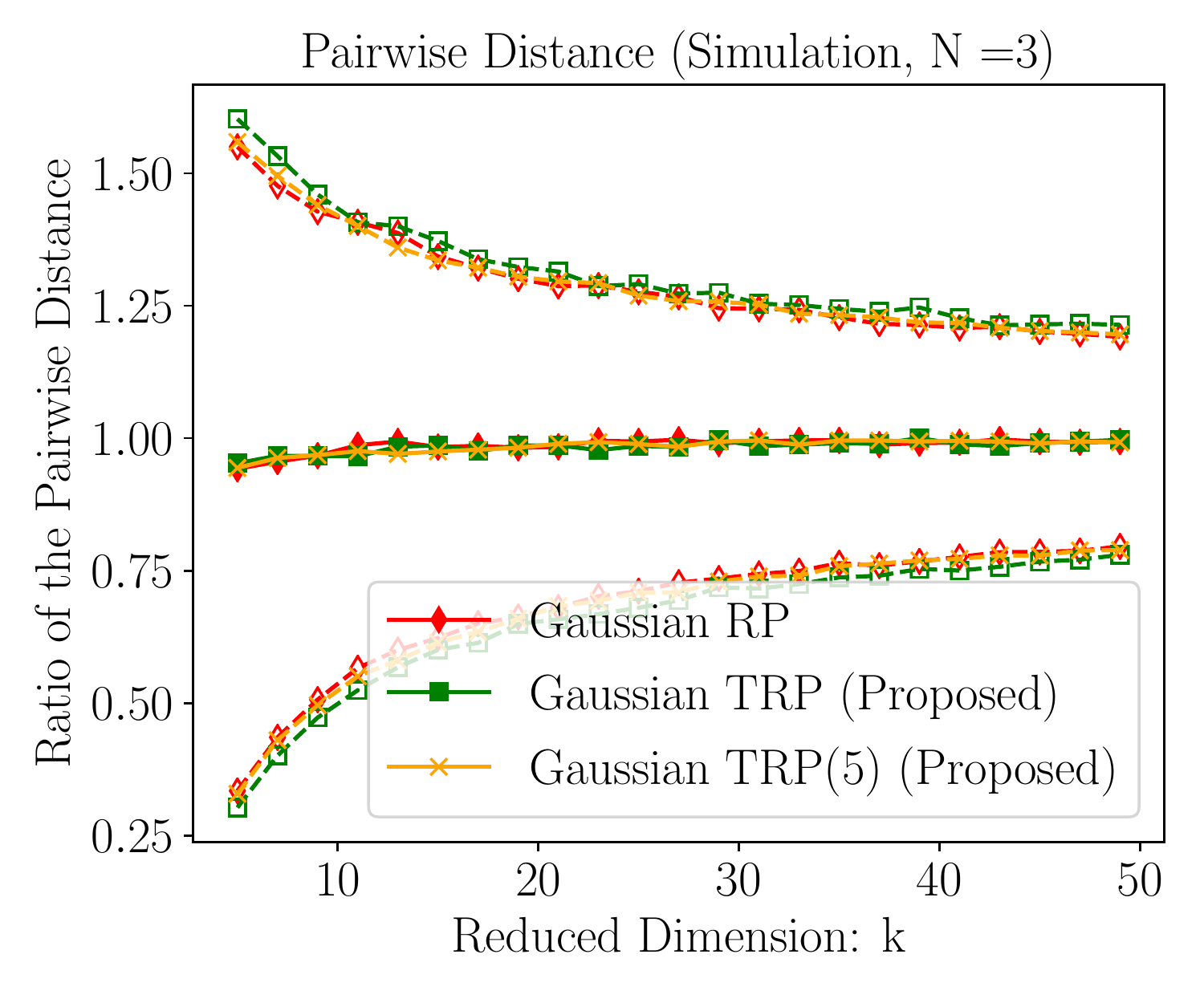}
	\end{subfigure}
	\begin{subfigure}{0.32\textwidth}
		\includegraphics[scale = 0.29]{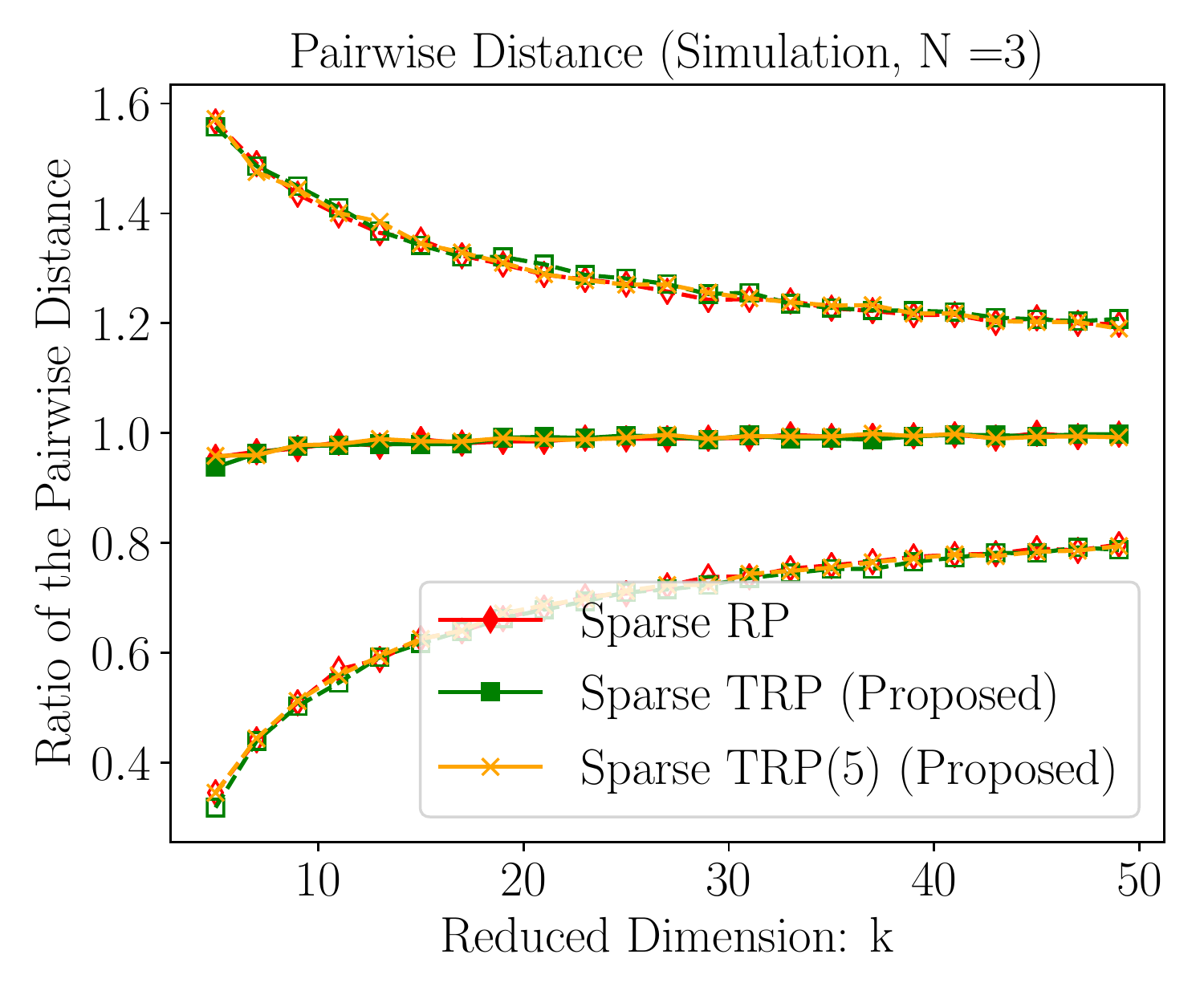}
	\end{subfigure}
	\begin{subfigure}{0.32\textwidth}
		\includegraphics[scale = 0.29]{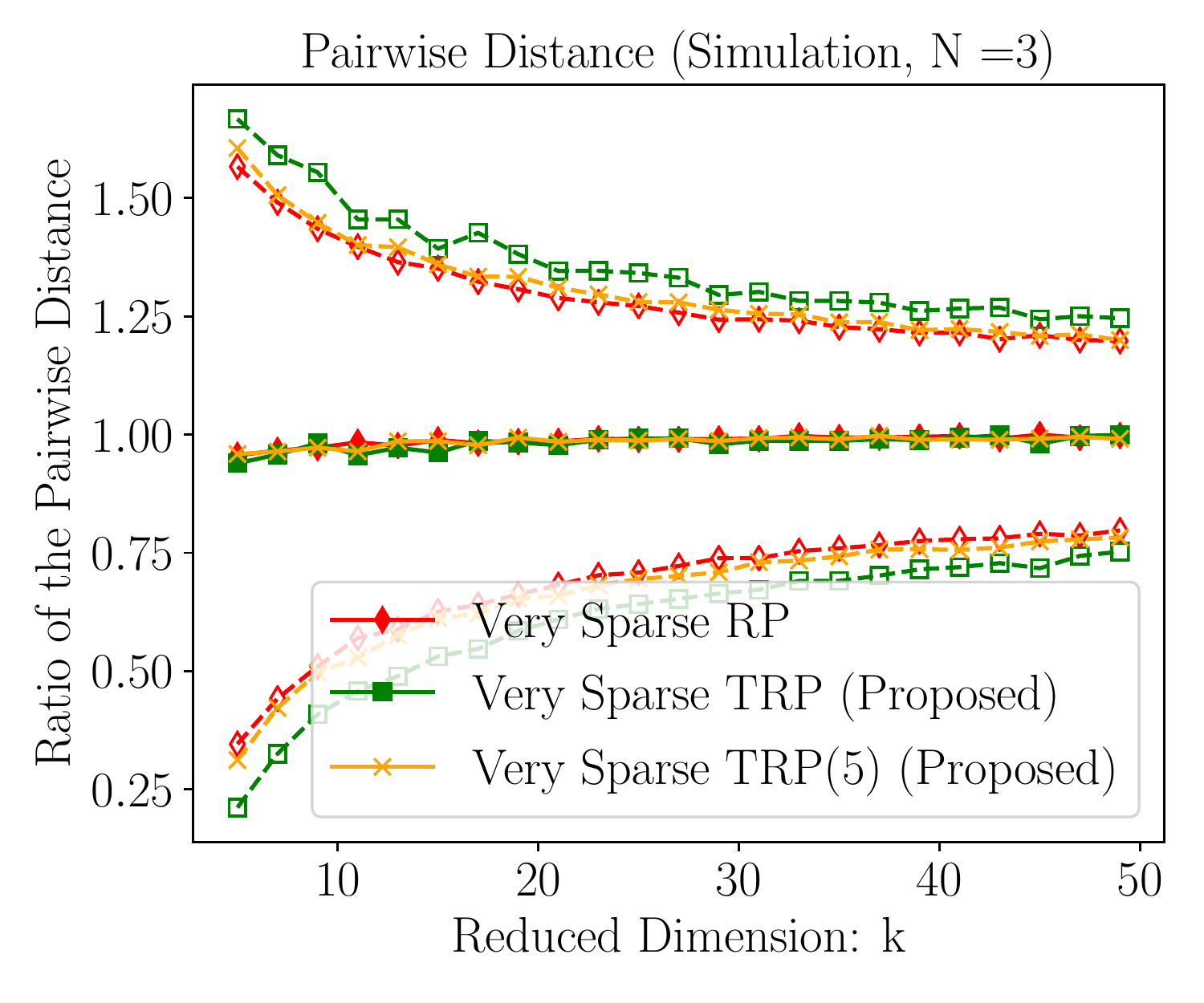}
	\end{subfigure}\\
	\caption{Average ratio of the pairwise distance for simulation data using: \textit{The plots correspond to the simulation for Gaussian, Sparase, Very Sparse RP, TRP, TRP(5) respectively with $n = 20, d = d_1d_2d_3 = 50 \times 50 \times 50 = 125000$ and each data vector comes from $N(\mathbf{0}, \mathbf{I})$. The dashed line represents the error bar 2 standard deviation away from the average ratio.}} 
	\label{fig:triple_krao}
\end{figure*}

\paragraph{Pairwise Cosine Similarity Estimation} 
The second experiment is to estimate the pairwise cosine similarity, i.e. $\frac{\mathbf{x}_i \cdot \mathbf{x}_j}{\|\mathbf{x}_i\|_2 \|\mathbf{x}_j\|_2}$ for $\mathbf{ x}_i, \mathbf{x}_j$. We use both the simulation data ($d = 10000$) and the MNIST data ($d = 784, n = 60000$). We experiment with Gaussian, Sparse, Very Sparse RP, TRP, and TRP(5) with the same setting as above ($k = 50$). We evaluate the performance by the average root mean square error (RMSE). The results is given in Table  \ref{tbl:mnist_inner_prod}, \ref{tbl:sim_inner_prod}.  

\begin{table}[t]
\centering
\begin{tabular}{l|l|l|l}
       & Gaussian        & Sparse          & Very Sparse     \\ \hline
RP     & 0.1409 (0.0015) & 0.1407 (0.0013) & 0.1412 (0.0014) \\ \hline
TRP    & 0.1431 (0.0016) & 0.1431 (0.0015) & 0.1520 (0.0033) \\ \hline
TRP(5) & 0.1412 (0.0012) & 0.1411 (0.0015) & 0.1427 (0.0014)
\end{tabular}
\caption{RMSE for the estimate of the pairwise inner product of the simulation data ($d = 10000, k = 50, n = 100 $), where standard error is in the parentheses.
}\label{tbl:sim_inner_prod}
\end{table}

\section{Application: Sketching}
\label{appendix:sketching}
Beyond random projection, our novel TRP also has an important application in sketching. Sketching is an important technique to accelerate expensive computations with widespread applications, such as regression, low-rank approximation, and graph sparsification, etc. \cite{halko2011finding,woodruff2014sketching} 
The core idea behind sketching is to compress a large dataset, typically a matrix or tensor, into a smaller one by multiplying a random matrix. 
In this section, we will mainly focus on the low-rank matrix approximation problem. Consider a matrix $\mathbf{X} \in \mathbb{R}^{m \times d}$ with rank $r$, 
we want to find the best rank-$r$ approximation with the minimal amount of time. The most common method is the randomized singular value decomposition (SVD), whose underlying idea is sketching.

First, we compute the linear sketch $\mathbf{Z} \in \mathbb{R}^{m \times k}$ by $\mathbf{Z} =\mathbf{X}\mathbf{\Omega}$, where $\mathbf{\Omega} \in \mathbb{R}^{d \times r}$ is the random map. Then we compute the QR decomposition of $\mathbf{X}\mathbf{\Omega}$ by $\mathbf{Q}\mathbf{R} = \mathbf{Z}$, where $\mathbf{Q} \in \mathbb{R}^{m \times k}, \mathbf{R} \in \mathbb{R}^{r \times r}$. At the end, we project $\mathbf{X}$ onto the column space of $\mathbf{Q}$, and obtain the approximation $\hat{\mathbf{X}} = \mathbf{Q} \mathbf{Q}^\top \mathbf{X}$.  

With our TRP, we can significantly reduce the storage of the random map, while achieving similar rate of convergence as demonstrated in Figure \ref{fig:col_matrix}. 
hm \ref{alg:var-red-structure-sketching}. And we will delay the theoretical analysis of this method for future works.

\begin{algorithm}[H]
	\caption{Tensor Sketching with Variance Reduction}\label{alg:var-red-structure-sketching}
	\begin{algorithmic}[1]
		\Require $\mathbf{X} \in \mathbb{R}^{m \times d}$, where $d = \prod_{i=1}^N d_n$ and 
		\rm{RMAP} is a user-specified function that generates a random dimension reduction map. $T$ is the number of runs for variance reduction averaging. 
		\Function{SSVR}{$\mathbf{X}, \{d_n\}, k, T, \rm{RMAP}$}
		\For{$t= 1 \dots T$}
		\For{$i = 1 \dots N$} 
		$\mathbf{\Omega}_i^{(t)} = \rm{RMAP}(d_i, k)$
		\EndFor
		\State $\mathbf{\Omega}^{(t)} = \mathbf{\Omega}_1^{(t)} \odot \cdots \odot \mathbf{\Omega}_{N}^{(t)}$
		\State $(\mathbf{Q}^{(t)}, \sim ) = \rm{QR}(\mathbf{X}\mathbf{\Omega}^{(t)})$ 
		\State $\hat{\mathbf{X}}^{(t)} = \mathbf{Q}^{(t)}\mathbf{Q}^{(t)T}\mathbf{X}$
		\EndFor
		\State
		$\hat{\mathbf{X}} = \frac{1}{T}\sum_{t=1}^T \hat{\mathbf{X}}^{(t)}$
		\State \Return $\mathbf{G}$
		\EndFunction
	\end{algorithmic}
\end{algorithm}

Furthermore, the extension of TRP to tensor data is also natural. To be specific, the $n^{th}$ unfolding of a large tensor $\mathscr{X} \in \mathbb{R}^{I_1 \times \cdots \times I_N}$, denoted as $\mathbf{X}^{(n)}$, has dimension $I_n \times I_{(-n)}$, where $I_{(-n)} = \prod_{i \neq n, i \in [N]} I_i$ . To construct a sketch for the unfolding, we need to create a random matrix of size $ I_{(-n)} \times k$. Then, our TRP becomes a natural choice to avoid the otherwise extremely expensive storage cost. For many popular tensor approximation algorithms, it is even necessary to perform sketching for every dimension of the tensor \cite{de2000multilinear,wang2015fast}. 
In the simulation section, we perform experiments for the unfolding of the higher-order order tensor with our structured sketching algorithms (Figure \ref{fig:col_matrix}). For more details in tensor algebra, please refer to \cite{kolda2009tensor}.

\paragraph{Experimental Setup}
In sketching problems, considering a $N$-D tensor $\mathscr{X} \in \mathbb{R}^{I^N}$ with equal length along all dimensions, we want to compare the performance of the low rank approximation with different maps for its first unfolding $\mathbf{X}^{(1)} \in \mathbb{R}^{I \times I^{N-1}}$. 
	
We construct the tensor $\mathscr{X}$ in the following way. Generate a core tensor $\mathscr{C} \in \mathbb{R}^{r^N}$, with each entry $\rm{Unif}([0,1])$. Independently generate $N$ orthogonal arm matrices by first creating $\mathbf{A}_1, \dots, \mathbf{A}_N \in \mathbb{R}^{r \times I}$ and then computing the arm matrices by $(\mathbf{Q}_n, \sim) = \rm{QR}(\mathbf{A}_n)$, for $1 \leq n \leq N$.
\begin{equation}
\mathscr{X} = \mathscr{C} \times_1 \mathbf{Q}_1 \cdots \times_N \mathbf{Q}_N + \sqrt{\frac{0.01 \cdot \|\mathscr{X}^\natural\|_F^2}{I^N}} \mathcal{N}(0,1). \nonumber
\end{equation}
Then, we construct the mode-1 unfolding of $\mathbf{X} = \mathbf{X}^{(1)}$, which has a rank smaller than or equal to $r$. 

In our simulation, we consider the scenarios of 2-D ($900 \times 900$), 3-D ($400 \times 400 \times 400$), 4-D ($100 \times 100 \times 100 \times 100$) tensor data, with corresponding mode-1 unfolding of size $900 \times 900$, $400 \times 160000$, $100 \times 1000000$ respectively and $r = 5$. In each scenario, we compare the performance for Gaussian RP, TRP, and TRP(5) maps with varying $k$ from 5 to 25. The TRP map in these scenarios has 2, 4, 6 components of size $30 \times k$, $20 \times k$, $10 \times k$ respectively. And the number of runs variance reduction averaging is $T = 5$. In the end, we evaluate the performance by generating the random matrix 100 times and compute the relative error $\frac{\|\mathbf{X} - \hat{\mathbf{X}}\|}{\|\mathbf{X}\|}$, and constructing a 95\% confidence interval for it.

\paragraph{Result}From Figure \ref{fig:col_matrix}, we can observe that the relative error decreases as $k$ increases as expected for all dimension reduction maps. The difference of the performance between the Khatri-Rao map and Gaussian map is small when $N = 2$, but increases when $N$ increases, whereas the Khatri-Rao variance reduced method is particularly effective producing strictly better performance than the other two. 

\begin{figure*}[ht!]
	\centering
	\begin{subfigure}{0.32\textwidth}
		\includegraphics[scale = 0.29]{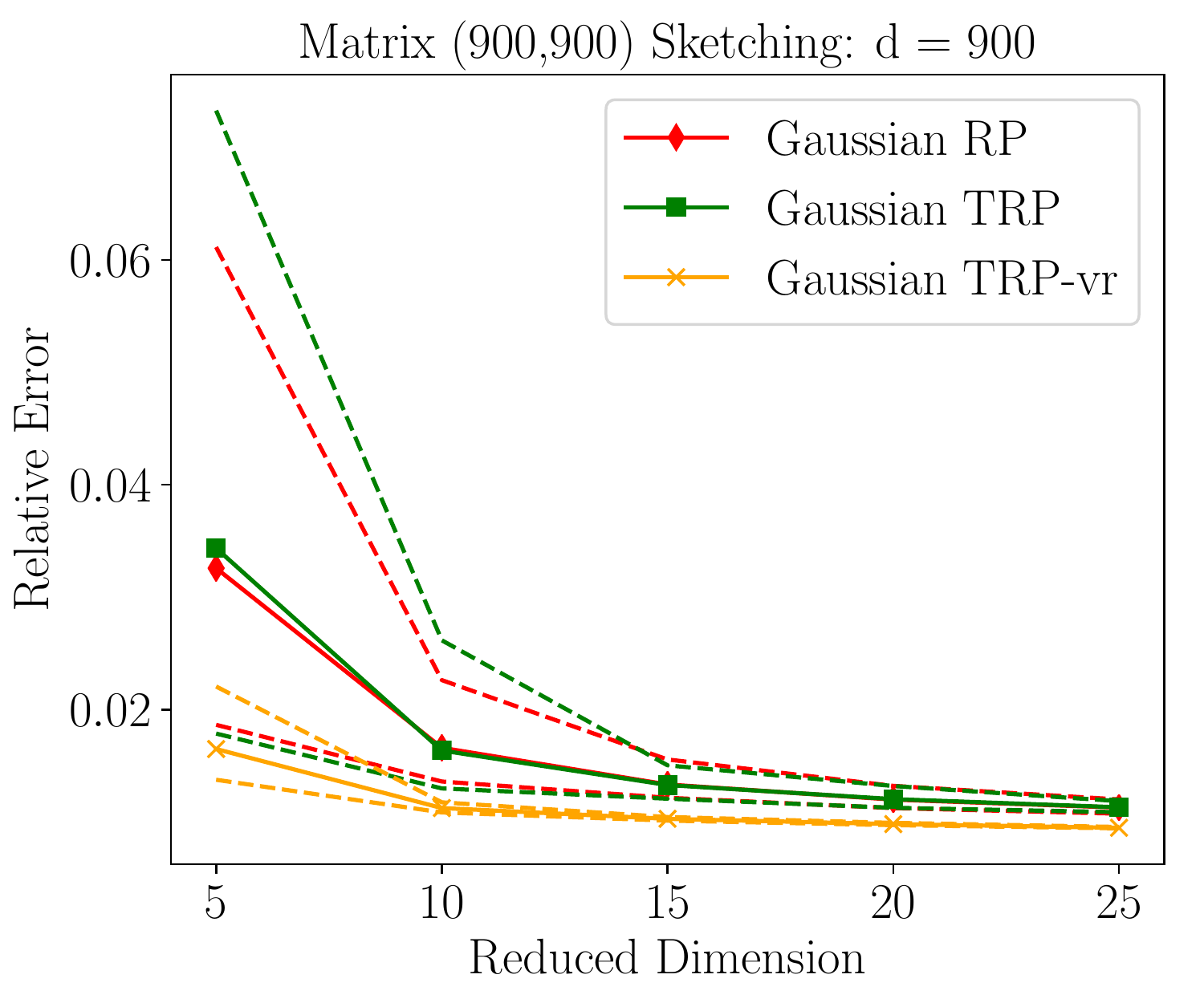}
	\end{subfigure}
	\begin{subfigure}{0.32\textwidth}
		\includegraphics[scale = 0.29]{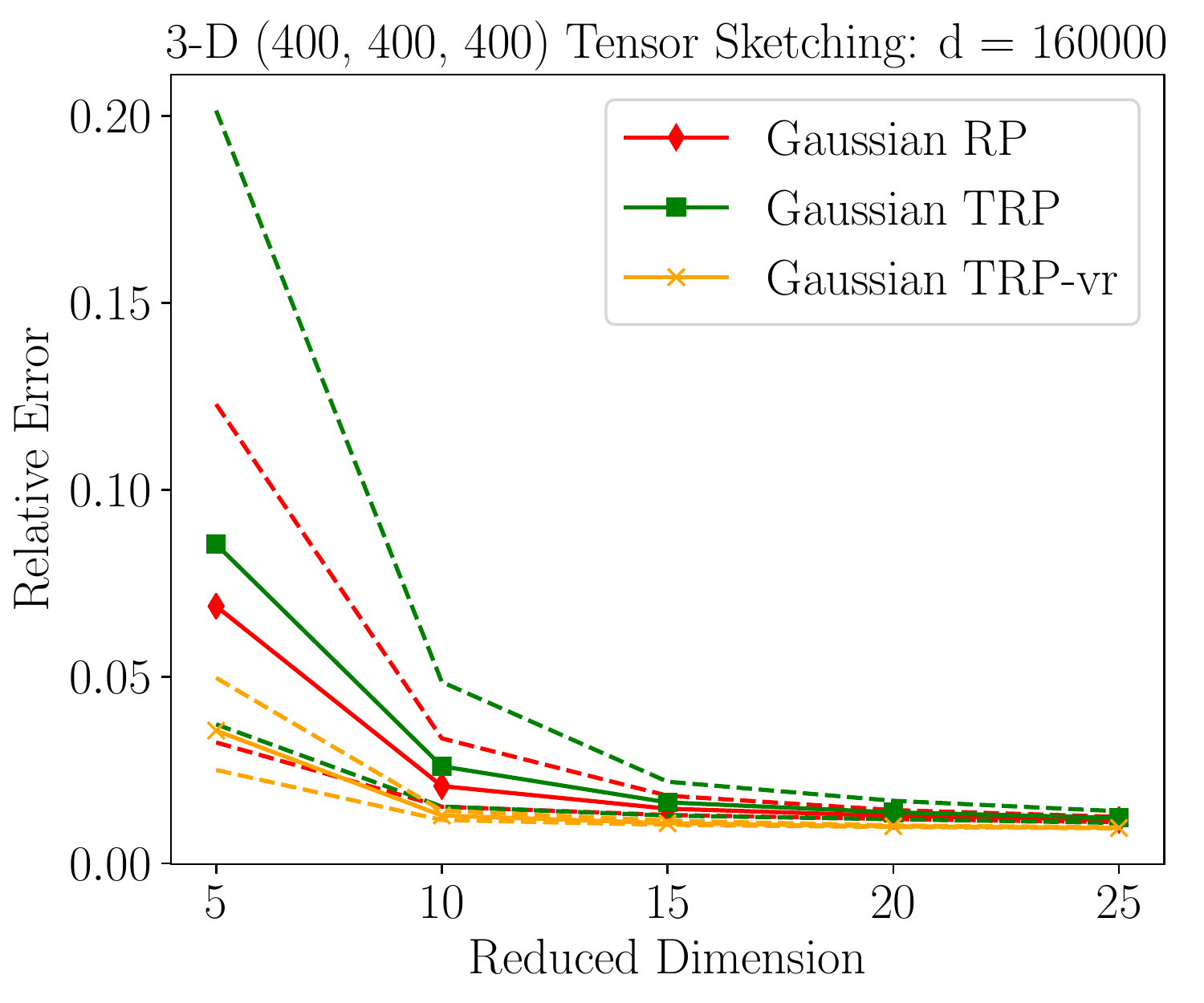}
	\end{subfigure}
	\begin{subfigure}{0.32\textwidth}
		\includegraphics[scale = 0.29]{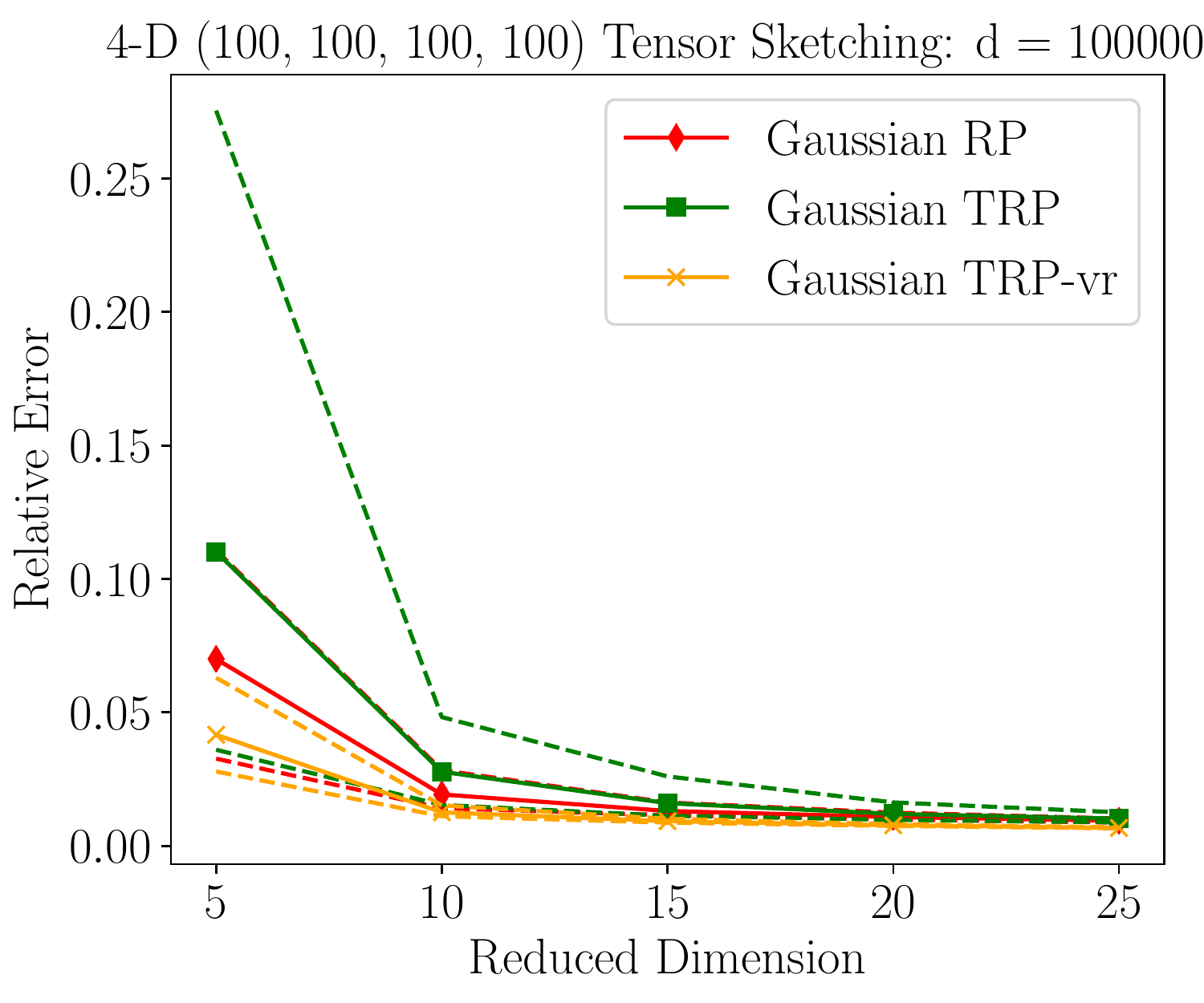}
	\end{subfigure}\\
	\caption{Relative Error for the low-rank tensor unfolding approximation: \textit{we compare the relative errors for low-rank tensor approximation with different input size: 2-D ($900 \times 900$), 3-D ($400 \times 400 \times 400$), 4-D ($100 \times 100 \times 100 \times 100$). In each setting, we compare the performance of Gaussian RP, TRP, and TRP(5). The dashed line stands for the 95\% confidence interval.}}
	\label{fig:col_matrix}
\end{figure*}

\section{Technical Lemmas}
In this section, we list some technical lemmas used in this paper.
These lemmas concern the tail probability of sub-Gaussian or generalized sub-exponential variables.

\begin{definition}
	\label{def:sub-gaussian}
	A random variable $x$ is called sub-Gaussian if $\mathbb{E} |x|^p = \mathcal{O}(p^{p/2})$ when $p\rightarrow \infty$. With this, we define sub-Gaussian norm for $x$ (less than infinity) as
	\begin{equation}
	\|x\|_{\varphi_2} = \sup_{p\ge 1} p^{-1/2} (\mathbb{E} |x|^p)^{1/p}.
	\end{equation}
\end{definition}

Note that for Bernoulli random variable, i.e., $\{-1,1\}$ with prob. $\{\frac{1}{2},\frac{1}{2} \}$,  $\varphi_2=1$; any bounded random variable with absolute value less than $M>0$ has $\varphi_2\le M$.  For standard Gaussian random variable, $\varphi_2=1$.

\begin{lem}
\label{lemma:hanson_wright}
(Hanson-Wright Inequality) Let $\mathbf{x} = (x_1,\cdots , x_n)\in \mathbb{R}^n$ be a random
vector with independent entries $x_i$ that satisfy $\mathbb{E} \mathbf{x}_i = 0$ and $\varphi_2(x_1)\le K$. Let $A$
be an $n\times n$ matrix. Then, for every $\eta \ge 0$, there exists a general constant $c$ s.t.
\begin{equation}
\mathbb{P}\left(|\mathbf{x^\top A x} - \mathbb{E}\mathbf{x^\top A x}|\ge \eta\right)\le 2\exp\left[-c\min\left\{\frac{\eta}{K^2\|A\|}, \frac{\eta^2}{\|A\|_F^2 K^4}\right\}\right]. \nonumber
\end{equation}
\end{lem}

\begin{proof}
	Please refer to \cite{rudelson2013hanson}
\end{proof}

\begin{lem}
\label{lemma:hanson-wright-sub-exponential}
Let $\mathbf{x}$ be a random
variable whose tail probability satisfies for every $\eta \ge 0$, there exists a constant $c_1$ s.t.
\begin{equation}
\mathbb{P}\left(|x|\ge \eta\right)\le 2\exp\left[-c_1\min\left(\eta,  \eta^2\right)\right]. \nonumber
\end{equation}
Then for any $k\ge 1$, $x$ satisfies generalized sub-exponential moment condition \ref{def:generalized-sub-exponential-mc} with $\alpha = 1$, i.e.,
\begin{equation}
\mathbb{E} |x|^k \le (Ck)^k, \nonumber
\end{equation}\label{eq:sub-exponential-moment-condition}
where $C= 1+\frac{1}{c_1}$.
\end{lem}
\begin{proof}
\begin{equation}
\begin{aligned}
&\mathbb{E} |x|^k =  \int_{0}^1 kx^{k-1} 2\exp[-c_1x^2]dx + \int_{1}^\infty kx^{k-1} 2\exp[-c_1x]dx \\
&\le 1+ + \frac{1}{c_1^k}\int_{0}^\infty ky^{k-1} 2\exp[-y]dy\\
&= 1+ \frac{1}{c_1^k} k\Gamma(k-1)\le \left[1+\frac{1}{c_1^k}\right] k^k.
\end{aligned}
\end{equation}
Noticing $\left[1+\frac{1}{c_1^k}\right]^{1/k} \le 1+\frac{1}{c_1}$, we finish the proof.
\end{proof}


\end{appendices}
\bibliographystyle{alpha}

\end{document}